\documentclass[submission,copyright,creativecommons]{eptcs}
\providecommand{\event}{TARK 2025} 
\usepackage{tikz}
\newcommand{\invG}{\mathsf{G}_{\mathsf{inv}}}
\newcommand{\biG}{\mathsf{biG}}
\newcommand{\KinvG}{\mathbf{K}\mathsf{G}_{\mathsf{inv}}}

\newcommand{\coimplies}{\Yleft}

\newcommand{\KGsquare}{\mathbf{K}\mathsf{G}^2}
\newcommand{\KG}{\mathbf{K}\mathsf{G}}
\newcommand{\KbiG}{\mathbf{K}\mathsf{biG}}
\newcommand{\pspace}{\mathsf{PSPACE}}
\newcommand{\Prop}{\mathtt{Prop}}
\newcommand{\Lit}{\mathtt{Lit}}
\newcommand{\invol}{{\sim}}
\newcommand{\bimodalLinv}{\mathscr{L}^{\invol}_{\Box,\lozenge}}

\newcommand{\Fmsf}{\mathsf{F}}
\newcommand{\Gmsf}{\mathsf{G}}

\newcommand{\Rmsf}{\mathsf{R}}

\newcommand{\Tmsf}{\mathsf{T}}

\newcommand{\Bmc}{\mathcal{B}}

\newcommand{\Imc}{\mathcal{I}}

\newcommand{\Omc}{{\mathcal{O}}}
\newcommand{\Pmc}{{\mathcal{P}}}

\newcommand{\cmc}{\mathcal{c}}
\newcommand{\hmc}{\mathcal{h}}
\newcommand{\lmc}{\mathcal{l}}
\newcommand{\tmc}{\mathcal{t}}

\newcommand{\Ffrak}{\mathfrak{F}}

\newcommand{\Mfrak}{{\mathfrak{M}}}
\newcommand{\Nfrak}{{\mathfrak{N}}}

\newcommand{\tmbf}{\mathbf{t}}

\newcommand{\Nmbb}{\mathbb{N}}

\newcommand{\Tmbb}{\mathbb{T}}

\newcommand{\nmbb}{\mathbb{n}}

\newcommand{\TKinvG}{\mathcal{T}\!(\KinvG)}
\newcommand{\real}{\mathsf{rl}}
\newcommand{\Var}{\mathsf{Var}}

\newcommand{\Str}{\mathsf{Str}}

\newcommand{\valueterm}{\mathsf{VT}}
\newcommand{\LF}{\mathsf{LF}}
\newcommand{\relterm}{\mathsf{RT}}
\newcommand{\zero}{\mathbb{0}}
\newcommand{\one}{\mathbb{1}}
\newcommand{\onetop}{\mathbf{1}}
\newcommand{\zerobot}{\mathbf{0}}
\newcommand{\WorldLabels}{\mathsf{WL}}
\usepackage[labelformat=original]{caption}
\usepackage{xcolor}
\usepackage[bb=boondox]{mathalfa}
\usepackage{lipsum} 
\usepackage{boondox-cal}
\usepackage{enumitem}
\usepackage[all]{xy}
\usepackage{stmaryrd}
\usepackage{iftex}
\usepackage{mathrsfs}
\usepackage{amsmath}
\usepackage{amsthm}
\usepackage{amssymb}
\usepackage{mathtools}
\usepackage{longtable}
\usepackage{xcolor}
\usepackage{forest}
\forestset{smullyan tableaux/.style={for tree={math content},where n children=1{!1.before computing xy={l=\baselineskip},!1.no edge}{},closed/.style={label=below:$\times$},},}
\ifpdf
 \usepackage{underscore} 
 \usepackage[T1]{fontenc} 
\else
 \usepackage{breakurl} 
\fi
\usepackage{xfrac}
\usepackage{wasysym}
\usepackage[normalem]{ulem}
\usepackage{thm-restate}

\newtheorem{lemma}{Lemma}
\theoremstyle{definition}
\newtheorem{definition}{Definition}
\theoremstyle{remark}

\newtheorem{example}{Example}
\newtheorem{convention}{Convention}
\title{Modal Logic for Reasoning About Uncertainty and Confusion}
\author{Marta B\'ilkov\'a\thanks{The first author is supported by the project Logical Structure of Information Channels, no.\ 21-23610M of the Czech Science Foundation. We are also grateful to the reviewers for their suggestions that helped to improve the paper.}
\institute{The Czech Academy of Sciences, Institute of Philosophy, Prague, Czech Republic}
\email{bilkova@cs.cas.cz}
\and
Thomas M.\ Ferguson
\institute{Department of Cognitive Science, Rensselaer Polytechnic Institute, Troy, USA}
\email{tferguson@gradcenter.cuny.edu}
\and
Daniil Kozhemiachenko
\institute{Aix Marseille Univ, CNRS, LIS, Marseille, France}
\email{daniil.kozhemiachenko@lis-lab.fr}
}
\def\titlerunning{Modal Logic for Reasoning About Uncertainty and Confusion}
\def\authorrunning{M.\ B\'ilkov\'a, T.~M.\ Ferguson, D.\ Kozhemiachenko}
\begin{document}
\allowdisplaybreaks
\maketitle
\begin{abstract}
We consider a~modal logic that can formalise statements about uncertainty and beliefs such as \emph{I think that my wallet is in the drawer rather than elsewhere} or \emph{I am confused whether my appointment is on Monday or Tuesday}. To do that, we expand G\"{o}del modal lo\-gic $\KG$ with the involutive negation~$\invol$ defined as $v(\invol\phi,w)=1-v(\phi,w)$. We provide semantics with the finite model property for our new logic that we call $\KinvG$ and show its equivalence to the standard semantics over $[0,1]$-valued Kripke models. Namely, we show that $\phi$ is valid in the standard semantics of $\KinvG$ iff it is valid in the new semantics. Using this new semantics, we construct a~constraint tableaux calculus for $\KinvG$ that allows for an explicit extraction of countermodels from complete open branches and then employ the tableaux calculus to obtain the $\pspace$-completeness of the validity in $\KinvG$.
\end{abstract}
\section{Introduction\label{sec:introduction}}
When people evaluate their uncertainty or beliefs, they can do it in two ways: either quantitatively, by assigning a specific number to the value of the belief, e.g., \emph{‘the chance of rain today is $73\%$’} or qualitatively, by comparing degrees of two beliefs without explicitly mentioning their numerical values as in \emph{‘I think that my wallet is more likely to be in the drawer than in the bag’}. Usually, a~person does not assign exact values to their beliefs and thus reasons qualitatively. Qualitative reasoning can be formalised, e.g., in G\"{o}del modal logic $\KG$ or, if one interprets ‘rather’ as ‘strictly more likely’, its expansion with $\triangle$ (Baaz' Delta~\cite{Baaz1996}) or $\coimplies$ (co-implication~\cite{Rauszer1974,Gore2000}) $\KbiG$ (modal bi-G\"{o}del logic).

Comparing degrees of beliefs is possible in $\KbiG$ because the semantics of its propositional fragment~--- bi-G\"{o}del logic ($\biG$) --- is characterised by the G\"{o}del t-norm $\wedge_\Gmsf$ and its residuum $\rightarrow_\Gmsf$ and their duals $\vee_\Gmsf$ and $\coimplies_\Gmsf$ defined on $[0,1]$ by:
\begin{align}\label{equ:Gtnorm}
a\wedge_\Gmsf b&=\min(a,b)&
a\rightarrow_\Gmsf b&=\begin{cases}1&\text{ if }a\leq b\\b&\text{ otherwise}\end{cases}&a\vee_\Gmsf b&=\max(a,b)&a\coimplies_\Gmsf b&=\begin{cases}0&\text{ if }a\leq b\\a&\text{ otherwise}\end{cases}
\end{align}
Propositional G\"{o}del logic is a prominent many-valued logic whose semantics defined on $[0,1]$  is suitable for formalising relative comparisons of degrees of truth. G\"{o}del \emph{modal} logics expand propositional G\"{o}del logic with modalities $\Box$ and $\lozenge$ and are usually interpreted on $[0,1]$-valued Kripke frames. There, propositional variables in states are evaluated with numbers from $[0,1]$, the propositional connectives correspond to G\"{o}del semantics shown in~\eqref{equ:Gtnorm}, and the values of $\Box\phi$ and $\lozenge\phi$ are defined as, respectively, the infimum and supremum of the values of $\phi$ in the accessible states.

Note that it is natural not only to compare degrees of certainty in two different events but also to assert whether the statement is \emph{more likely to be true} or \emph{more likely to be false}. This can be expressed by statements such as \emph{‘my wallet is likely in the drawer’}, meaning \emph{‘I think that my wallet is in the drawer \underline{rather than elsewhere}’}. I.e., the (subjective) likelihood of the wallet being in the drawer is greater than~$\frac{1}{2}$. Furthermore, an agent may be \emph{confused} about a~statement $p$, i.e., not be able to discern $p$ with its negation.

As $\biG$ can express only $0$ and $1$, the first kind of statements cannot be formalised in $\KbiG$. However, $\frac{1}{2}$ is expressible with the involutive negation~$\invol$ defined as $v(\invol\phi,w)=1-v(\phi,w)$: $v(p\leftrightarrow\invol p,w)=1$ iff $v(p,w)=\frac{1}{2}$. Moreover, since $\neg p$ --- the G\"{o}delian negation of $p$ --- is defined as $p\rightarrow_\Gmsf\zerobot$, $p$ and $\neg p$ can never have the same value, and thus, can always be distinguished. I.e., the second kind of statements is not formalisable either.

In this paper, we will present a~modal logic that allows us to express such statements. Let us now provide a~broader context of our work.

\paragraph{G\"{o}del logics with involution}
The modal logic that we consider in our paper expands $\invG$ --- the propositional G\"{o}del logic with involutive negation $\invol$ that was introduced in~\cite{EstevaGodoHajekNavara2000}. Adding $\invol$ greatly enhances the expressivity of the G\"{o}del logic. Not only can $\frac{1}{2}$ be expressed, but also coimplication (pseudo-difference, interpreted as $\phi$ excludes $\chi$) and the Baaz Delta (the $1$-detecting operator) can be defined.
\begin{align*}
\phi\coimplies\chi&\coloneqq\invol(\invol\chi\rightarrow\invol\phi)&\triangle\phi&\coloneqq\onetop\coimplies(\onetop\coimplies\phi)
\end{align*}

We observe briefly that $\invol$ is closer than $\neg$ to the intuitive reading of ‘not’ in contexts involving truth degrees. Indeed, if the truth degree of $p$ (‘Peter is tall’) is $\tfrac{1}{2}$, it is reasonable to assume that the truth degree of $\invol p$ (Peter is not tall) is also $\tfrac{1}{2}$. Similarly, if Peter is shorter than average (i.e., the truth degree of $p$ is smaller than $\tfrac{1}{2}$), one may safely accept that the truth degree of ‘not-$p$’ should be greater than $\tfrac{1}{2}$. On the other hand, the truth degree of $\neg p$ would be $0$ whenever the truth degree of $p$ is positive. Thus, using $\neg p$ to stand for ‘Peter is not tall’ is counterintuitive as it says that ‘Peter is tall is contradictory'.

Note, moreover, that $\invG$ is \emph{paraconsistent} in the following sense: (1) $(p\wedge\invol p)\rightarrow q$ is not valid; (2) $p,\invol p\not\models_{\invG}q$ if entailment is defined via the preservation of order on $[0,1]$. These properties were investigated in~\cite{ErtolaEstevaFlaminioGodoNoguera2015,ConiglioEstevaGispertGodo2021}.

Paraconsistency of the propositional fragment is a~useful property when it comes to the analysis of beliefs. Indeed, from a~classical standpoint, \emph{all contradictions are equivalent} (because they are all universally false). Hence, $\Box(p\wedge\invol p)\leftrightarrow\Box(q\wedge\invol q)$ is valid in every classical modal logic. If one interprets~$\Box$ as a~belief modality, this means that the agent cannot differentiate between contradictions. This, however, is counterintuitive since a~person can (implicitly) believe in one contradiction but not the other or reject one contradiction with more conviction than the other. In paraconsistent logics, however, contradictions are \emph{not} equivalent, and thus, are not indistinguishable under $\Box$.

\paragraph{G\"{o}del modal logics}
Modal expansions of G\"{o}del logics and their applications have been extensively studied. The $\Box$ and $\lozenge$ fragments\footnote{Note that $\Box$ and $\lozenge$ are not interdefinable in G\"{o}del modal logic.} of $\KG$ were axiomatised in~\cite{CaicedoRodriguez2010}. Hypersequent calculi were constructed in~\cite{MetcalfeOlivetti2009,MetcalfeOlivetti2011} and used to obtain the $\pspace$-completeness of both fragments. $\KG$ in the bi-modal language was axiomatised in~\cite{CaicedoRodriguez2015,RodriguezVidal2021} (for fuzzy and crisp\footnote{A frame $\langle W,R\rangle$ is called \emph{fuzzy} if $R:W\times W\rightarrow[0,1]$ and \emph{crisp} if $R:W\times W\rightarrow\{0,1\}$.} frames, respectively). Moreover, it was shown in~\cite{CaicedoMetcalfeRodriguezRogger2013,CaicedoMetcalfeRodriguezRogger2017} that crisp and fuzzy $\KG$ are $\pspace$-complete. 

The applications of G\"{o}del modal logics and its expansions are well researched. In~\cite{RodriguezTuytEstevaGodo2022}, the completeness of $\mathbf{K45}$ and $\mathbf{KD45}$ extensions of $\KG$ with respect to non-normalised and normalised possibilistic frames was established. In~\cite{AguileraDieguezFernandez-DuqueMcLean2022,AguileraDieguezFernandez-DuqueMcLean2022KR}, a~tem\-po\-ral logic expanding the G\"{o}del logic with co-implication was proposed. In addition, paraconsistent expansions of $\KG$ that can be used to reason about contradictory beliefs were proposed and studied in~\cite{BilkovaFrittellaKozhemiachenko2022IJCAR,BilkovaFrittellaKozhemiachenko2023IGPL,BilkovaFrittellaKozhemiachenko2024JLC}.

\paragraph{G\"{o}del description logics}
G\"{o}del description logics (DLs) were proposed in~\cite{BobilloDelgadoGomez-RamiroStraccia2009} to represent graded information in the ontologies and were further studied in~\cite{BobilloDelgadoGomez-RamiroStraccia2012}. Just as the classical description logics are notational variants of the classical logics with the global (universal) modality, G\"{o}del DLs can be considered a notational variant of global G\"{o}del modal logics. Moreover, G\"{o}del DLs are usually equipped with the involutive negation. They differ from \L{}ukasiewicz and Product fuzzy DLs because they are decidable and often have the same complexity as their classical counterparts even in the cases of expressive logics~\cite{BorgwardtDistelPenaloza2014DL,BorgwardtDistelPenaloza2014KR,Borgwardt2014PhD,BorgwardtPenaloza2017}. Note, however, that since infinite-valued G\"{o}del DLs lack the finite model property, their decision procedures do not produce explicit (counter-)models.

\paragraph{Contributions and plan of the paper}
In this paper, we consider $\KinvG$~--- an expansion of $\KG$ with involutive negation and construct a~tableaux calculus that allows for the extraction of finite counter-models from failed proofs and can be used to establish its $\pspace$-completeness. Our contribution is two-fold. First, following~\cite{CaicedoMetcalfeRodriguezRogger2013,CaicedoMetcalfeRodriguezRogger2017}, we present an alternative semantics for $\KinvG$ that possesses the finite model property. In particular, we show that the sets of validities with respect to the standard and the alternative semantics coincide (both in the case of crisp and arbitrary frames). Second, using this semantics, we construct a~tableaux calculus by combining the techniques from~\cite{Haehnle1994} and~\cite{Rogger2016phd} that allows us to read off countermodels from complete open branches. We then use this calculus to establish $\pspace$-completeness of~$\KinvG$.

The remainder of the text is structured as follows. In Section~\ref{sec:KinvG}, we present $\KinvG$ and its standard semantics. We also discuss the formalisation of reasoning about uncertainty and confusion in $\KinvG$. Section~\ref{sec:Fmodels} is dedicated to the alternative semantics with the finite model property. In Section~\ref{sec:tableaux}, we construct a~tableaux calculus for $\KinvG$ and show its soundness and completeness. Then, in Section~\ref{sec:complexity}, we use the tableaux to obtain the $\pspace$-completeness of $\KinvG$. Finally, we summarise our results and provide a~plan for future work in Section~\ref{sec:conclusion}.
\section{Language and standard semantics\label{sec:KinvG}}
Let us present the language and standard semantics of $\KinvG$. We fix a~countable set $\Prop$ of propositional variables and define the language $\bimodalLinv$ using the grammar below.
\begin{align*}
\bimodalLinv\ni\phi&\Coloneqq p\in\Prop\mid\invol\phi\mid(\phi\wedge\phi)\mid(\phi\rightarrow\phi)\mid\Box\phi\mid\lozenge\phi
\end{align*}

The next definitions introduce the notions of frames and $\KinvG$ models.
\begin{definition}[Frames]\label{def:frames}~
\begin{itemize}[noitemsep,topsep=2pt]
\item A \emph{fuzzy frame} is a tuple $\Ffrak=\langle W,R\rangle$ with $W\neq\varnothing$ and $R:W\times W\rightarrow[0,1]$.
\item A \emph{crisp frame} is a tuple $\Ffrak=\langle W,R\rangle$ with $W\neq\varnothing$ and $R:W\times W\rightarrow\{0,1\}$.
\end{itemize}
For a~frame $\Ffrak=\langle W,R\rangle$ and $w\in W$, we set $R(w)=\{w':wRw'>0\}$.
\end{definition}
\begin{definition}[Semantics of $\KinvG$]\label{def:KinvG}
A \emph{$\KinvG$-model} is a tuple $\Mfrak=\langle W,R,v\rangle$ with $\langle W,R\rangle$ being a frame and $v:\Prop\times W\rightarrow[0,1]$ (a $\KinvG$-valuation) extended to the complex formulas as follows (cf.~\eqref{equ:Gtnorm} for the definition of $\wedge_\Gmsf$ and $\rightarrow_\Gmsf$).
\begin{align*}
v(\invol\phi,w)&=1\!-\!v(\phi,w)&v(\phi\wedge\chi,w)&=v(\phi,w)\wedge_\Gmsf v(\chi,w)\\
v(\phi\rightarrow\chi,w)&=v(\phi,w)\rightarrow_\Gmsf v(\chi,w)&v(\Box\phi,w)&=\inf\limits_{w'\in W}\{wRw'\!\rightarrow_\Gmsf\!v(\phi,w')\}\\
&&v(\lozenge\phi,w)&=\sup\limits_{w'\in W}\{wRw'\!\wedge_\Gmsf\!v(\phi,w')\}
\end{align*}

We say that $\phi\!\in\!\bimodalLinv$ is \emph{$\KinvG$-valid on a pointed frame $\langle\Ffrak,\!w\rangle$} ($\Ffrak,\!w\!\models_{\KinvG}\!\phi$) iff $v(\phi,w)=1$ for any model $\Mfrak$ on $\Ffrak$. $\phi$ is \emph{$\KinvG$-valid on frame $\Ffrak$} ($\Ffrak\!\models_{\KinvG}\!\phi$) iff $\Ffrak,w\!\models_{\KinvG}\!\phi$ for any $w\in\Ffrak$. Finally, $\phi$ is \emph{$\KinvG$-valid} (or simply \emph{valid}) iff $\Ffrak\models_{\KinvG}\phi$ for every $\Ffrak$.
\end{definition}
\begin{convention}\label{conv:connectives}
Given a formula $\phi$, we use $\lmc(\phi)$ to denote the number of occurrences of symbols in~$\phi$. We will write $\phi\leftrightarrow\chi$ as a shorthand for $(\phi\rightarrow\chi)\wedge(\chi\rightarrow\phi)$ and use the following defined connectives:
\begin{align*}
\onetop&\coloneqq p\rightarrow p&\zerobot&\coloneqq\invol\onetop&\neg\phi&\coloneqq\phi\rightarrow\zerobot\nonumber\\
\phi\vee\chi&\coloneqq\invol(\invol\phi\wedge\invol\chi)&\phi\coimplies\chi&\coloneqq\invol(\invol\chi\rightarrow\invol\phi)&\triangle\phi&\coloneqq\onetop\coimplies(\onetop\coimplies\phi)
\end{align*}
One can see from~\eqref{equ:Gtnorm} and Definition~\ref{def:KinvG} that $v(\phi\leftrightarrow\chi,w)=1$ if $v(\phi,w)=v(\chi,w)$ and $\min(v(\phi,w),v(\chi,w))$, otherwise; $v(\phi\vee\chi,w)\!=\!v(\phi,w)\vee_\Gmsf v(\chi,w)$; $v(\phi\coimplies\chi,w)\!=\!v(\phi,w)\coimplies_\Gmsf v(\chi,w)$; and $v(\triangle\phi,w)\!=\!1$ if $v(\phi,w)\!=\!1$ and $v(\triangle\phi,w)\!=\!0$, otherwise.
\end{convention}

Let us now look at how we can use $\KinvG$ to formalise the contexts we considered in the previous section. First, we consider the assertion that $\phi$ is more likely to be true than false.
\begin{example}\label{example:wallet}
Consider the following statement from the \nameref{sec:introduction}.
\begin{description}[noitemsep,topsep=2pt]
\item[$\mathsf{wal.}$:] \textit{I~think that my wallet is in the drawer rather than elsewhere.}
\end{description}
We write $d$ for ‘my wallet is in the drawer’ and use $\Box$ for ‘I think that’. Now, we need to formalise that the degree of subjective likelihood of $d$ is greater than~$\frac{1}{2}$. That is, we should write a formula that is true at $w$ when $v(\Box d,w)>\frac{1}{2}$. We formalise $\mathsf{wal.}$ as follows: $\phi_\mathsf{wal.}\coloneqq\invol\triangle(\Box d\rightarrow\invol\Box d)$. By Definition~\ref{def:KinvG}, and Convention~\ref{conv:connectives}, one can see that $v(\phi_\mathsf{wal.},w)=1$ iff $v(\Box d\rightarrow\invol\Box d,w)<1$. This means that $v(\Box d,w)>v(\invol\Box d,w)$, i.e., $v(\Box d,w)>1-v(\Box d,w)$, whence, $v(\Box d,w)>\tfrac{1}{2}$, as required.
\end{example}

Let us now see how to formalise situations involving confusion using $\KinvG$.
\begin{example}\label{example:confusion}
Assume that Ann wants to come to Brittney's birthday party. Unfortunately, the only piece of information Ann has about the party is that it is held this weekend and lasts one evening. Thus, Ann cannot discern between the birthday party happening on Saturday ($s$) or Sunday ($\invol s$).

Formally, this means that in every accessible state, the values of $s$ and $\invol s$ must be the same. Thus, we write the following formula: $\phi_\mathsf{bd}\coloneqq\Box\triangle(s\leftrightarrow\invol s)$. It is easy to verify that $v(\phi_\mathsf{bd},w)=1$ iff $v(s,w')=\frac{1}{2}=v(\invol s,w')$ in every $w'\in R(w)$. Indeed, the value of $\triangle(s\leftrightarrow\invol s)$ always belongs to $\{0,1\}$. Thus, $v(\phi_\mathsf{bd},w)=1$ iff $v(\triangle(s\leftrightarrow\invol s),w')=1$ in every accessible state. That is, $v(s\leftrightarrow\invol s,w')=1$, and hence, $v(s,w)=v(\invol s,w)=1-v(s,w)$.
\end{example}

We finish the section with the following observations.  First, $\Box p\leftrightarrow\invol\lozenge\invol p$ is valid on~$\Ffrak$ iff $\Ffrak$ is crisp. Second, $\Box$ and $\lozenge$ are not interdefinable on the class of all frames.
\begin{restatable}{proposition}{Boxlozengeinterdefinability}\label{proposition:Boxlozengeinterdefinability1}
$\Ffrak\models_{\KinvG}\Box p\leftrightarrow\invol\lozenge\invol p$ iff $\Ffrak$ is crisp.
\end{restatable}
\begin{proof}
Observe that on crisp frames, the semantics of modalities can be simplified as follows: $v(\Box\phi,w)=\inf\{v(\phi,w'):wRw'\}$ and $v(\lozenge\phi,w)=\sup\{v(\phi,w'):wRw'\}$. From here, it is evident that $\Box p\leftrightarrow\invol\lozenge\invol p$ is valid on crisp frames. For the converse, assume that $\Ffrak$ is not crisp, i.e., there are $w,w'\in\Ffrak$ s.t.\ $wRw'=x$ for some $0<x<1$. Now for every $w''\in W$, set $v(p,w'')=wRw''$ and $v(p,w'')=1$ only if $wRw''=1$. It is clear that $v(\Box p,w)=1$. On the other hand, we have that $v(\lozenge\invol p,w)>0$ since $0<v(\invol p,w')\leq1-x<1$ and $wRw'=x$. But then, $v(\invol\lozenge\invol p,w)\neq1$, as required.
\end{proof}
\begin{restatable}{proposition}{Boxlozengenoninterdefinability}\label{proposition:Boxlozengeinterdefinability2}
$\Box$ and $\lozenge$ are not interdefinable on the class of all frames. That is, there is a~model $\Mfrak$ and $w\in\Mfrak$ s.t.\ (1) there is no $\lozenge$-free formula $\chi$ s.t.\ $v(\chi,w)=v(\lozenge p,w)$; (2) there is no $\Box$-free formula $\psi$ s.t.\ $v(\psi,w)=v(\Box p,w)$.
\end{restatable}
\begin{proof}
Consider Fig.~\ref{fig:Boxlozengenoninterdefinable} and note that $v(\Box p,w)\!=\!\tfrac{1}{5}$ and $v(\lozenge p,w)\!=\!\tfrac{1}{4}$. We show by induction on $\tau$ that
\begin{align}
\forall\tau\in\bimodalLinv:v(\tau,w')\in\left\{0,\sfrac{1}{5},\sfrac{4}{5},1\right\}\text{ and }v(\tau,w'')\in\left\{0,\sfrac{1}{4},\sfrac{3}{4},1\right\}\label{equ:ab}
\end{align}
The basis cases in~\eqref{equ:ab} hold by the construction of the model and the cases of propositional connectives are obtained by applying the induction hypothesis. Finally, as $R(w')=R(w'')=\varnothing$, for $\tau=\lozenge\sigma$, we have $v(\lozenge\sigma,w')=v(\lozenge\sigma,w'')=0$; and for $\tau=\Box\sigma$, we have that $v(\Box\sigma,w')=v(\Box\sigma,w'')=1$.

Then, we can show for every $\tau\in\bimodalLinv$ that
\begin{align}\label{equ:Boxlozengeinterdefinability}
v(\tau,w'')=0&\Rightarrow v(\tau,w')=0&v(\tau,w'')=\sfrac{1}{4}&\Rightarrow v(\tau,w')=\sfrac{1}{5}\nonumber\\
v(\tau,w'')=\sfrac{3}{4}&\Rightarrow v(\tau,w')=\sfrac{4}{5}&v(\tau,w'')=1&\Rightarrow v(\tau,w')=1
\end{align}
and
\begin{align}\label{equ:Boxlozengeinterdefinabilityback}
v(\tau,w')=0&\Rightarrow v(\tau,w'')=0&v(\tau,w')=\sfrac{1}{5}&\Rightarrow v(\tau,w'')=\sfrac{1}{4}\nonumber\\
v(\tau,w')=\sfrac{4}{5}&\Rightarrow v(\tau,w'')=\sfrac{3}{4}&v(\tau,w')=1&\Rightarrow v(\tau,w'')=1
\end{align}
For~\eqref{equ:Boxlozengeinterdefinability}, we proceed by induction. The basis case of $\tau=p$ holds by the construction of the model. If $\tau=\lozenge\sigma$, then $v(\lozenge\sigma,w')=v(\lozenge\sigma,w'')=0$. If $\tau=\Box\sigma$, then $v(\Box\sigma,w')=v(\Box\sigma,w'')=1$.

Let $\tau=\varrho\wedge\sigma$. The cases of $v(\tau,w'')\in\{0,1\}$ are straightforward. Assume that $v(\varrho\wedge\sigma,w'')=\frac{1}{4}$, then w.l.o.g.\ $v(\varrho,w'')=\frac{1}{4}$ and $v(\sigma,w'')\in\{\frac{1}{4},\frac{3}{4},1\}$. By the induction hypothesis, we have that $v(\varrho,w'')=\frac{1}{5}$ and $v(\sigma,w'')\in\{\frac{1}{5},\frac{4}{5},1\}$, whence, $v(\varrho\wedge\sigma,w')=\frac{1}{5}$. The case when $v(\varrho\wedge\sigma,w'')=\frac{3}{4}$ is tackled in the same way.

Consider $\tau=\varrho\rightarrow\sigma$. We deal only with the most instructive case of $v(\varrho\rightarrow\sigma,w'')=\frac{1}{4}$. Here, we have $v(\sigma,w'')=\frac{1}{4}$ and $v(\varrho,w'')\in\{1,\frac{3}{4}\}$. By the induction hypothesis, we have $v(\sigma,w')=\frac{1}{5}$ and $v(\varrho,w')\in\{\frac{4}{5},1\}$. Hence, $v(\varrho\rightarrow\sigma,w')=\frac{1}{5}$.

Other propositional connectives can be dealt with similarly.

For~\eqref{equ:Boxlozengeinterdefinabilityback}, the case of $\tau=p$ holds by construction; if $\tau=\Box\sigma$, then $v(\Box\sigma,w')=v(\Box\sigma,w'')=1$; if $\tau=\lozenge\sigma$, then $v(\lozenge\sigma,w')=v(\lozenge\sigma,w'')=0$. The cases of propositional connectives can be dealt with in the same way as for~\eqref{equ:Boxlozengeinterdefinability}.

We now show by induction on formulas $\chi$ and $\psi$ that (i) there is no $\Box$-free formula $\chi$ s.t.\ $v(\chi,w)\in\left\{\frac{1}{5},\frac{4}{5}\right\}$ and (ii) there is no $\lozenge$-free formula $\psi$ s.t.\ $v(\psi,w)\in\left\{\frac{1}{4},\frac{3}{4}\right\}$.

For (i), the basis case holds by the construction of the model. The cases of propositional connectives can be established by a straightforward application of the induction hypothesis. Now let $\chi=\lozenge\tau$. It is clear that $v(\lozenge\tau,w)\neq\frac{4}{5}$ since $wRw'=wRw''=\frac{2}{3}$. We check that $v(\lozenge\tau,w)\neq\frac{1}{5}$. Assume for contradiction that $v(\lozenge\tau,w)=\frac{1}{5}$. Then, $v(\tau,w')=\frac{1}{5}$ from~\eqref{equ:ab}. But using~\eqref{equ:Boxlozengeinterdefinability}, we have that $v(\tau,w'')=\frac{1}{4}$, whence $v(\lozenge\tau,w)=\frac{1}{4}$. Contradiction.

For (ii), the basis case holds by the construction of the model and the cases of propositional connectives can be proven by the application of the induction hypothesis. We show that $v(\Box\tau,w)\notin\left\{\frac{1}{4},\frac{3}{4}\right\}$. Observe that $wRw'=wRw''=\frac{2}{3}$, whence, $v(\Box\tau,w)\!\neq\!\frac{3}{4}$. We are going to check that $v(\Box\tau,w)\!\neq\!\frac{1}{4}$. The proof is similar to the one of (i) but we use~\eqref{equ:Boxlozengeinterdefinabilityback}. If $v(\Box\tau,w)=\frac{1}{4}$, then $v(\tau,w'')=\frac{1}{4}$ by~\eqref{equ:ab}. But from here, we obtain $v(\tau,w')=\frac{1}{5}$ by~\eqref{equ:Boxlozengeinterdefinabilityback}. Contradiction.
\end{proof}
\begin{figure}
\begin{align*}
\xymatrix{w'':p=\sfrac{1}{4}~&&\ar[ll]|{\sfrac{2}{3}}~w:p=0~\ar[rr]|{\sfrac{2}{3}}&&~w':p=\sfrac{1}{5}}
\end{align*}
\caption{All variables have the same values as $p$.}
\label{fig:Boxlozengenoninterdefinable}
\end{figure}
\section{Semantics for $\KinvG$ with the finite model property\label{sec:Fmodels}}
It is well-known (cf., e.g.,~\cite{CaicedoRodriguez2010}) that the standard semantics of $\KG$ (modal G\"{o}del logic) and thus, of $\KinvG$\footnote{As G\"{o}del negation is expressible in $\bimodalLinv$ (Convention~\ref{conv:connectives}), $\KinvG$ is a~conservative extension of $\KG$ with~$\invol$.}, lack the finite model property (FMP). Still, semantics over so-called $\Fmsf$-models ($\Fmsf$ stands for ‘finitary’) satisfying the finite model property can be provided to any expansion of $\KG$ with order-based connectives (i.e., connectives expressible via infima and suprema of lattice operations)~\cite{CaicedoMetcalfeRodriguezRogger2017}.

One can see, however, that $\invol$ is not an order-based connective. Thus, we cannot use the general result of~\cite{CaicedoMetcalfeRodriguezRogger2017} to provide a~semantics over $\Fmsf$-models. Still, we can adapt the approach from~\cite{CaicedoMetcalfeRodriguezRogger2013}. Namely, we will redefine the models in such a~way that in a~given state $w$, there is only a~finite set of values that formulas of the form $\Box\phi$ or $\lozenge\phi$ are allowed to have. Then modal formulas will be ‘witnessed’ in $w$ if there is an accessible state $w'$ where the value of $\phi$ is sufficiently close to the value of $\Box\phi$ or $\lozenge\phi$.
\begin{definition}[$\Fmsf$-models]\label{def:F-KinvG}
An \emph{$\Fmsf$-model} is a tuple $\Mfrak=\langle W,R,T,v\rangle$ with $\langle W,R\rangle$ being a~frame and $T:W\rightarrow\Pmc_{<\omega}([0,1])$\footnote{We use $\Pmc_{<\omega}([0,1])$ to denote the set of all \emph{finite} subsets of $[0,1]$.} be s.t.\ for all $w\in W$, it holds that: $\{0,\frac{1}{2},1\}\subseteq T(w)$, and if $x\in T(w)$, then $1-x\in T(w)$.
Finally, $v$ (\emph{valuation}) is a~map $v:\Prop\times W\rightarrow[0,1]$ that is extended to the complex formulas as in $\KinvG$ (Definition~\ref{def:KinvG}) in the cases of propositional connectives and in the modal cases, as follows:
\begin{align*}
v(\Box\phi,w)&=\max\{x\in T(w)\mid x\leq\inf\limits_{w'\in W}\{wRw'\rightarrow_\Gmsf v(\phi,w')\}\}\\
v(\lozenge\phi,w)&=\min\{x\in T(w)\mid x\geq\sup\limits_{w'\in W}\{wRw'\wedge_\Gmsf v(\phi,w')\}\}
\end{align*}
We say that $\phi\in\bimodalLinv$ is \emph{$\Fmsf$-valid} iff $v(\phi,w)=1$ in every $\Fmsf$-model $\Mfrak$ and every $w\in\Mfrak$.
\end{definition}

In the definition above, $T$ assigns a~finite subset of $[0,1]$ to each state in the model. This way, $T(w)$ is the set of values that modal formulas can have at $w$. Note also that the value of $\Box\phi$ in an $\Fmsf$-model can be interpreted as an approximation from below of its ‘real’ value in a~standard model. Dually, the value of $\lozenge\phi$ is an approximation from above. It is important to observe that we need the $1-x$ closure of $T(w)$ and the presence of $\frac{1}{2}$. Otherwise, we could have the following ‘models’.
\begin{align*}
\Nfrak:\xymatrix{w_0\ar^(.4){R=\sfrac{1}{2}}[r]&w_1:p=\frac{1}{2}}&&T(w_0)=\{0,1\}
&&\|\|&&\Nfrak':\xymatrix{w'_0\ar^(.4){R=1}[r]&w'_1:p=\frac{1}{3}}&&T(w'_0)=\{0,\sfrac{1}{3},\sfrac{1}{2},1\}
\end{align*}
One can see that $v(\Box p\wedge\lozenge\invol p,w_0)=1$ in~$\Nfrak$ which is impossible in $\KinvG$-models. Furthermore, $\Nfrak'$ is crisp but $v'(\Box p\leftrightarrow\invol\lozenge\invol p,w'_0)=0$. By Proposition~\ref{proposition:Boxlozengeinterdefinability1}, this is impossible crisp in $\KinvG$-models.

Let us now state the main result of the section.
\begin{restatable}{theorem}{semanticsequivalence}\label{theorem:semanticsequivalence}~
\begin{enumerate}[noitemsep,topsep=2pt]
\item $\phi$ is $\KinvG$-valid iff it is $\Fmsf$-valid on all \emph{finite models}.
\item $\phi$ is $\KinvG$-valid on \emph{crisp frames} iff it is $\Fmsf$-valid on all \emph{finite crisp models}.
\end{enumerate}
\end{restatable}
We adapt the proof of~\cite[Theorem~1]{CaicedoMetcalfeRodriguezRogger2013}. We begin with some technical notions.
\begin{definition}[Generated models]\label{def:generatedmodel}~
\begin{itemize}[noitemsep,topsep=2pt]
\item A frame $\widehat{\Ffrak}=\langle\widehat{W},\widehat{R}\rangle$ is \emph{a~subframe} of $\Ffrak=\langle W,R\rangle$ if $\widehat{W}\subseteq W$ and $\widehat{R}$ is the restriction of $R$ on $\widehat{W}$.
\item $\widehat{\Mfrak}=\langle\widehat{W},\widehat{R},\widehat{v}\rangle$ is \emph{a~submodel of} $\Mfrak=\langle W,R,v\rangle$ if $\langle\widehat{W},\widehat{R}\rangle$ is a~subframe of $\langle W,R\rangle$ and $\widehat{v}$ is the restriction of $v$ on $\widehat{W}$.
\item $\widehat{\Mfrak}=\langle\widehat{W},\widehat{R},\widehat{T},\widehat{v}\rangle$ is \emph{a~submodel of} $\Mfrak=\langle W,R,T,v\rangle$ if $\langle\widehat{W},\widehat{R}\rangle$ is a~subframe of $\langle W,R\rangle$ and $\widehat{v}$ and $\widehat{T}$ are the restrictions of $v$ and $T$ on $\widehat{W}$.
\item $\widehat{\Mfrak}$ is \emph{generated by $X\subseteq W$} if it is the smallest submodel containing $X$ s.t.\ if $w\in\widehat{W}$ and $wRw'>0$, then $w'\in\widehat{W}$, as well.
\end{itemize}
\end{definition}
\begin{convention}~
\begin{itemize}[noitemsep,topsep=2pt]
\item A model $\Mfrak=\langle W,R,v\rangle$ (or $\Mfrak=\langle W,R,T,v\rangle$) is \emph{tree-like} if $\langle W,R\rangle$ is a~directed rooted tree.
\item Given a~tree-like model $\Mfrak$, $\hmc(\Mfrak)$ denotes its \emph{height}.
\end{itemize}
\end{convention}
First, we can establish the analogue of~\cite[Lemma~1]{CaicedoMetcalfeRodriguezRogger2013}. We show that to verify the validity of a~formula (with respect to $\KinvG$-models and $\Fmsf$-models), it suffices to consider only tree-like models of \emph{finite height}. The proof is standard, so we omit it here.
\begin{lemma}\label{prop:Lemma1}
Let $\Mfrak=\langle W,R,v\rangle$ and $\widehat{\Mfrak}=\langle\widehat{W},\widehat{R},\widehat{v}\rangle$ be $\KinvG$-models. Let further $\Mfrak^\Fmsf=\langle W^\Fmsf,R^\Fmsf,T,v^\Fmsf\rangle$ and $\widehat{\Mfrak^\Fmsf}=\langle\widehat{W^\Fmsf},\widehat{R^\Fmsf},\widehat{T},\widehat{v^\Fmsf}\rangle$ be $\Fmsf$-models. Then the following statements hold.
\begin{enumerate}[noitemsep,topsep=2pt]
\item[$(a)$] If $\widehat{\Mfrak}$ (resp., $\widehat{\Mfrak^\Fmsf}$) is a generated submodel of $\Mfrak$ ($\Mfrak^\Fmsf$), then $\widehat{v}(\phi,w)\!=\!v(\phi,w)$ for all $w\in\widehat{W}$ (resp., $\widehat{v^\Fmsf}(\phi,w)=v^\Fmsf(\phi,w)$ for all $w\in\widehat{W^\Fmsf}$) and $\phi\in\bimodalLinv$.
\item[$(b)$] For all $\phi\in\bimodalLinv$ and $w\in\Mfrak$ ($w\in\Mfrak^\Fmsf$), there is a tree-like submodel $\widehat{\Mfrak}$ of~$\Mfrak$ ($\widehat{\Mfrak^\Fmsf}$ of~$\Mfrak^\Fmsf$) generated by $w$ s.t.\ $v(\phi,w)=\widehat{v}(\phi,w)$ ($v^\Fmsf(\phi,w)=\widehat{v^\Fmsf}(\phi,w)$) and $\hmc$$(\widehat{\Mfrak})\leq\lmc(\phi)$.
\item[$(c)$] Let $g:[0,1]\rightarrow[0,1]$ be s.t.\ $g(0)=0$, $g(1)=1$, \underline{$g(1-x)=1-g(x)$}, and $x\leq x'$ iff $g(x)\leq g(x')$\footnote{Observe that this, in particular, means that $x>x'$ implies $g(x)>g(x')$.}. Let further $\Nfrak=\langle W_\Nfrak,R_\Nfrak,v_\Nfrak\rangle$ (resp., $\Nfrak^\Fmsf=\langle W^\Fmsf_\Nfrak,R^\Fmsf_\Nfrak,T_\Nfrak,v^\Fmsf_\Nfrak\rangle$) be s.t.\ $W_\Nfrak=W$, $wR_\Nfrak w'=g(wRw')$, and $v_\Nfrak(p,w)=g(v(p,w))$ ($W^\Fmsf_\Nfrak=W^\Fmsf$, $wR^\Fmsf_\Nfrak w'=g(wR^\Fmsf w')$, $T_\Nfrak(w)=g(T(w))$, $v^\Fmsf_\Nfrak(p,w)=g(v^\Fmsf(p,w))$) for every $w,w'\in W$ ($w,w'\in W^\Fmsf$) and $p\in\Prop$. Then it holds that $v_\Nfrak(\phi,w)=g(v(\phi,w))$ (resp., $v^\Fmsf_\Nfrak(\phi,w)=g(v^\Fmsf(\phi,w))$) for every $\phi\in\bimodalLinv$ and $w\in W$ ($w\in W^\Fmsf$). 
\end{enumerate}
\end{lemma}
We remark quickly that because of $\invol$, we need the underlined part of $(c)$ that was not present in the original statement. Otherwise, we could have $g(\frac{1}{2})=\frac{2}{3}$ which would fail $(c)$ since $v_\Mfrak(p\leftrightarrow\invol p,w)=1$ if $v_\Mfrak(p,w)=\frac{1}{2}$ but $v_\Nfrak(p\leftrightarrow\invol p,w)=\frac{1}{3}$ since $v_\Nfrak(p,w)=\frac{2}{3}$.

The next statements are analogous to~\cite[Lemmas~2 and~3]{CaicedoMetcalfeRodriguezRogger2013}.
\begin{restatable}{lemma}{Lemmatwo}\label{lemma:Lemma2}
For any tree-like $\Fmsf$-model $\Mfrak=\langle W,R,T,v\rangle$ of finite height with root $x_0$ there is a tree-like $\KinvG$-model $\widehat{\Mfrak}=\langle \widehat{W},\widehat{R},\hat{v}\rangle$ with root $\hat{x}_0$ s.t.\ $v(\phi,x_0)=\hat{v}(\phi,\hat{x}_0)$ for each $\phi\in\bimodalLinv$. In addition, if $\Mfrak$ is crisp, so is $\widehat{\Mfrak}$.
\end{restatable}
\begin{proof}
The statement may be proven by induction on $\hmc(\Mfrak)$. The case in which $\hmc(\Mfrak)=0$ serves as basis, setting $\widehat{W}=W$, $\widehat{R}=\varnothing$, and $\widehat{v}=v$; in this case, the statement follows trivially. As induction step, suppose that $\hmc(\Mfrak)=n+1$. By the induction hypothesis, for each $y\in R(x_0)$, the submodel of $\Mfrak$ generated by $\lbrace y\rbrace$~--- the model $\Mfrak_{y}=\langle W_{y},R_{y},v_{y}\rangle$~--- there is a $\KinvG$ tree model $\widehat{\Mfrak}=\langle\widehat{W_y},\widehat{R_y},\widehat{v}_y\rangle$ such that $\hat{v}_y(\phi,\hat{y})=v_y(\phi,y)=v(\phi,y)$. 

Take $T(x_{0})=\lbrace \alpha_{0},\ldots,\alpha_{m}\rbrace$ and for each $k\in\Nmbb$, pick an order embedding $h_{k}:[0,1]\rightarrow[0,1]$ such that $h_{k}(0)=0$, $h_{k}(1)=1$, and $h(1-j)=1-h(j)$ for all $j\in[0,1]$ such that:
\begin{align*}
h_{k}(\alpha_{i})=\alpha_{i} & \mbox{ for all }i\leq m\mbox{ and }k\in\Nmbb\\
h_{k}[(\alpha_{i},\alpha_{i+1})]=(\alpha_{i},\min(\alpha_{i}+\sfrac{1}{k},\alpha_{i+1})) & \mbox{ for all }i\leq m-1\mbox{ and even }k\in\Nmbb\\
h_{k}[(\alpha_{i},\alpha_{i+1})]=(\max(\alpha_{i}-\sfrac{1}{k},\alpha_{i+1}),\alpha_{i+1}) & \mbox{ for all }i\leq m-1\mbox{ and odd }k\in\Nmbb
\end{align*}
For each $y\in R(x_0)$ and $k\in\Nmbb$, there exists an $\Fmsf$-model $\widehat{\Mfrak}^{k}_{y}$ defined so that $\widehat{W}_{y}^{k}$ includes for each $\hat{x}_{y}\in \widehat{W}_{y}$ a~point $\hat{x}_{y}^{k}$ with root $\hat{y}^{k}$ and for all $\hat{x}_{y}$, $\hat{z}_{y}$, and $\phi\in\bimodalLinv$, $\hat{x}_{y}^{k}\widehat{R}_{y}^{k}\hat{z}_{y}^{k}=h_k(\hat{x}_{y}\widehat{R}_{y}\hat{z}_{y})$ and $\widehat{V}_{y}^{k}(\psi,\hat{x}_{y}^{k})=h_k(\widehat{v}_{y}(\psi,\hat{x}_{y}))$. We can define a $\KinvG$-tree-model $\widehat{\Mfrak}$ where
\begin{align*}
\widehat{W}&=\bigcup\limits_{y\in R(x_0)}\bigcup\limits_{k\in\Nmbb} \widehat{W}_{y}^{k}\cup\lbrace\hat{x}_{0}\rbrace\\
x\widehat{R}z&=
\begin{cases}
h_{k}(x_0Ry) & \mbox{if }x=\hat{x}_{0}\mbox{ and }z=\hat{y}^{k}\mbox{ for some }y\in R(x_0)\mbox{ and }k\in\Nmbb\\
x\widehat{R}_{y}^{k}z & \mbox{if }x,z\in\widehat{W}_{y}^{k}\mbox{ for some }y\in R(x_0)\mbox{ and }k\in\Nmbb\\
0 & \mbox{otherwise}
\end{cases}\\
\widehat{v}(p,x)&=
\begin{cases}
v(p,x_{0}) & \mbox{if }x=x_{0}\\
\widehat{v}_{y}^{k}(p,x) & \mbox{if }x\in\widehat{W}_{y}^{k}\mbox{ for some }y\in R(x_0)\mbox{ and }k\in\Nmbb\\
\end{cases}
\end{align*}
The induction hypothesis ensures that for each $\phi\in\bimodalLinv$ and $\hat{x}_{y}^{k}\in \widehat{W}$ distinct from $\hat{x}_{0}$, $\widehat{v}_{y}^{k}(\phi,\hat{x}_{y}^{k})=\widehat{v}(\phi,\hat{x}_{y}^{k})$. What remains is then to prove that $\widehat{v}(\phi,\hat{x}_{0})=\widehat{v}(\phi,x_{0})$.

This is proven by induction on $\lmc(\phi)$. In case $\phi$ is an atom, this follows by construction of $\widehat{v}$. The non-modal connectives (including involutive negation) follow standardly. We consider the case of $\phi=\Box\psi$. 

Suppose that $\phi=\Box\psi$. Then in case that $v(\Box\psi,x_{0})=1$, for all $y\in R(x_0)$, $x_0Ry\leq v(\psi,y)$, whence by Lemma~\ref{prop:Lemma1}, $v(\psi,y)=v_y(\psi,y)=\widehat{v}_y(\psi,\hat{y})$. As for any $y$ $x_0Ry\leq \widehat{v}_y(\psi,\hat{y})$, it follows that for any $k$, $\hat{x}_0 \widehat{R}\hat{y}^k\leq \widehat{v}_y^k(\psi,\hat{y}^k)$, entailing that $\widehat{v}(\Box\psi,\hat{x}_0)=1$ as well. Suppose then that $v(\Box\psi,x_{0})=\alpha_i\neq 1$. Then for all $z\in W$, $x_0Rz\rightarrow_\Gmsf v(\psi,z)\geq\alpha_i$ and by construction of the order embeddings, we get the important feature that for all $z\in\widehat{W}$, $\hat{x}_0\widehat{R}z\rightarrow_\Gmsf\hat{v}(\psi,z)\geq\alpha_i$. We consider two cases. In case there is a $y_0\in W$ for which $x_0Ry\rightarrow_\Gmsf v(\psi,y)=\alpha_i$, then $x_0Ry_0>v(\psi,y_0)=\alpha_i$ and the order embeddings ensure that for all $k\in\Nmbb$, $\hat{x}_0\widehat{R}\hat{y}_0^k>\hat{v}(\psi,\hat{y}_0^k)=\alpha_i$. Otherwise, there must be some $y_0\in W$ for which $x_0Ry_0\rightarrow_\Gmsf v(\psi,y_0)\in (\alpha_i,\alpha_{i+1})$. By construction of the order embeddings, for any $\varepsilon >0$ there is some $k\in\Nmbb$ for which $h_k(x_0Ry_0\rightarrow_\Gmsf v(\psi,y_0))=\hat{x}_0 \widehat{R}\hat{y}_0^k\rightarrow_\Gmsf\hat{v}(\psi,\hat{y}_0^k)\in(\alpha_i,\alpha_i +\varepsilon)$. In either case, $\hat{v}(\Box\psi,\hat{x}_0)=\alpha_i=v(\Box\psi,x_0)$. 

Finally, we note that if $\Mfrak$ is crisp, so is each $\widehat{\Mfrak}_{y}$ and thus so is each $\widehat{\Mfrak}_{y}^{k}$, whence $\widehat{\Mfrak}$ is crisp.
\end{proof}
\begin{definition}[Fragments]\label{def:fragments}
A~\emph{fragment} is a~set $\Sigma\subseteq\!\bimodalLinv$ s.t.\ it is closed under taking subformulas and $\{\zerobot,\onetop\}\subseteq\Sigma$. Now let $\cdot^\Imc:\bimodalLinv\rightarrow\bimodalLinv$ be as follows:
\begin{align*}
\phi^\Imc&=\begin{cases}\invol\phi\mbox{ if there is no }\xi\in\!\bimodalLinv\mbox{ such that }\phi=\invol\xi\\
\xi\mbox{ if there is a }\xi\in\!\bimodalLinv\mbox{ such that }\phi=\invol\xi\end{cases}
\end{align*}

We say that $\Sigma\subseteq\bimodalLinv$ is an \emph{involutively closed fragment (ICF)} if it is a~fragment and for every $\phi,\chi\in\Sigma$, $\circ\in\{\wedge,\rightarrow\}$, and $\heartsuit\in\{\Box,\lozenge\}$, it holds that:
\begin{itemize}
\item if $\phi\in\Sigma$, then $\phi^\Imc\in\Sigma$;
\item if $\phi\circ\chi\in\Sigma$, then $\{\phi^\Imc\circ\chi,\phi\circ\chi^\Imc,\phi^\Imc\circ\chi^\Imc\}\subseteq\Sigma$;
\item if $\heartsuit\phi\in\Sigma$, then $\heartsuit\phi^\Imc\in\Sigma$.
\end{itemize}

Given an involutively closed fragment~$\Sigma$, we define:
\begin{align*}
\Sigma_\Box&=\{\Box\phi\mid\Box\phi\!\in\!\Sigma\}&\Sigma_\lozenge&=\{\lozenge\phi\mid\lozenge\phi\!\in\!\Sigma\}&\Sigma_\Lit&=\{l\mid l\!\in\!\Sigma\text{ and }l\text{ is a~literal}\}\\
\Sigma_\Box^\invol&=\{\invol\Box\phi\mid\invol\Box\phi\!\in\!\Sigma\}&\Sigma_\lozenge^\invol&=\{\invol\lozenge\phi\mid\invol\lozenge\phi\!\in\!\Sigma\}
\end{align*}
\end{definition}
\begin{restatable}{lemma}{Lemmathree}\label{lemma:Lemma3}
Let $\Sigma$ be a~finite ICF and $\Mfrak=\langle W,R,v\rangle$ a~tree-like $\KinvG$-model of finite height with root~$x_0$. Then there is a~tree-like $\Fmsf$-model $\widehat{\Mfrak}=\langle\widehat{W},\widehat{R},\widehat{T},\widehat{v}\rangle$ with root $x_0$ s.t.\ $\langle\widehat{W},\widehat{R}\rangle\subseteq\langle W,R\rangle$, $|W|\leq|\Sigma|^{\mathcal{h}(\Mfrak)}$, $|\widehat{T}(w')|\!\leq\!|\Sigma|$, and $v(\phi,x_0)\!=\!\widehat{v}(\phi,x_0)$ for every $\phi\!\in\!\Sigma$ and $w'\!\in\!\widehat{\Mfrak}$. In addition, if $\Mfrak$ is crisp, so is~$\widehat{\Mfrak}$.
\end{restatable}
\begin{proof}
The statement may be proven by induction on $\hmc(\Mfrak)$. As the basis, letting $\widehat{W}=W$, $\widehat{R}=\varnothing$, $\widehat{v}=v$, and $\widehat{T}(x_{0})=\lbrace 0,\frac{1}{2},1\rbrace$ establishes the statement immediately.

As induction step, suppose that $\hmc(\Mfrak)=n+1$. For each $y\in R(x_0)$, we note that the submodel $\Mfrak_{y}$ generated by $\lbrace y\rbrace$ is a $\KinvG$ model with $\hmc(\Mfrak_{y})\leq n$ and by induction hypothesis there exists an accompanying finite $\Fmsf$-tree-model $\widehat{\Mfrak}_{y}$ with root $y$ such that for all $\phi\in\Sigma$, $\widehat{v}_{y}(\phi,y)=v_{y}(\phi,y)=v(\phi,y)$. Additionally, by induction hypothesis, $\lvert\widehat{W}_{y}\rvert\leq\lvert\Sigma\rvert^{n}$ and $\lvert\widehat{T}_{y}(x)\rvert\leq\lvert\Sigma\rvert+1$ for each $x\in\widehat{W}_{y}$. We build an appropriate finite $\Fmsf$-model by selecting a set of appropriate $y\in R(x_0)$. 

For a $\Delta\subseteq\bimodalLinv$ let $v_{x}[\Delta]=\lbrace v(\phi,x)\mid \phi\in\Delta\rbrace$; the set $v_{x_{0}}[\Sigma_\Box\cup\Sigma_\Box^\invol\cup\Sigma_\lozenge\cup\Sigma_\lozenge^\invol]\cup\lbrace 0,\frac{1}{2},1\rbrace$ is a finite set $\lbrace \alpha_1,\ldots,\alpha_m\rbrace$ closed under $1-x$ where if $i<j$ then $\alpha_i<\alpha_j$. For each $\Box\phi\in\Sigma_\Box$ such that $v(\Box\phi,x_0)=\alpha_i<1$ and $\alpha_j=1-\alpha_{i}$, fix points $y_{\Box\phi},y_{\invol\Box\phi}\in R(x_0)$ for which $x_0Ry_{\Box\phi}\rightarrow_\Gmsf v(\phi,y_{\box\phi})<\alpha_{i+1}$ and $x_0Ry_{\invol\Box\phi}\rightarrow_\Gmsf v(\phi,y_{\invol\box\phi})<\alpha_{j+1}$. Similarly, for each $\lozenge\phi\in\Sigma_\lozenge$ such that $v(\lozenge\phi,x_0)=\alpha_i>0$ and $\alpha_{j}=1-\alpha_{i}$, fix points $y_{\lozenge\phi},y_{\invol\lozenge\phi}\in R(x_0)$ such that $\min(x_0Ry_{\lozenge\phi},v(\phi,y_{\lozenge\phi}))\geq \alpha_{i-1}$ and $\min(x_0Ry_{\invol\lozenge\phi},v(\phi,y_{\invol\lozenge\phi}))\geq \alpha_{j-1}$. Then let $Y=\lbrace y_{\psi}\mid \psi\in\Sigma_\Box\cup\Sigma_\Box^\invol\cup\Sigma_\lozenge\cup\Sigma_\lozenge^\invol\rbrace$.

Now, define an $\Fmsf$ model $\widehat{\Mfrak}=\langle\widehat{W},\widehat{R},\widehat{T},\widehat{v}\rangle$ such that
\begin{align*}
\widehat{W}&=\bigcup_{y\in W}\widehat{W}_{y}\cup\lbrace x_{0}\rbrace\\
x\widehat{R}z&=
\begin{cases}
x_0Rz & \mbox{if }x=x_{0}\mbox{ and }z\in R(x_0)\\
x\widehat{R}_{y}z & \mbox{if }x,z\in\widehat{W}_{y}\mbox{ for some }y\in Y\\
0 & \mbox{otherwise}
\end{cases}\\
\widehat{T}(x)&=
\begin{cases}
v_{x_{0}}[\Sigma_\Box\cup\Sigma_\Box^\invol\cup\Sigma_\lozenge\cup\Sigma_\lozenge^\invol]\cup\lbrace 0,\frac{1}{2},1\rbrace&\mbox{if }x=x_0\\
\widehat{T}_{y}(x) & \mbox{ if }x\in\widehat{W}_{y}\mbox{ for some }y\in Y
\end{cases}\\
\widehat{v}(p,x)&=
\begin{cases}
v(p,x_{0}) & \mbox{if }x=x_{0}\\
\widehat{v}_{y}(p,x)&\mbox{if }x\in\widehat{W}_{y}\mbox{ for some }y\in Y\\
\end{cases}
\end{align*}
Importantly, note that $\widehat{R}^{+}[x_{0}]=Y$ where $Y\subseteq R(x_0)$ and that for each $y\in Y$ and $\phi\in\bimodalLinv$, $x_0\widehat{R}y=x_0Ry$ and $\hat{v}(\phi,y)=v(\phi,y)$. Now, we prove that for all $\phi$, $\hat{v}(\phi,x_0)=v(\phi,x_0)$ by induction on length of $\phi$. The atomic case follows by construction of $\hat{v}$ and the non-modal connectives (including involutive negation) are standard. We thus prove the case in which $\phi=\Box\psi$.

In case $v(\Box\psi,x_0)=1$, for each $y\in R^+[x_0]$, $x_0Ry\rightarrow_\Gmsf v(\psi,y)=1$, whence for each such $y$, $x_0Ry\leq v(\psi,y)$. As $\widehat{R}^{+}[x_0]\subseteq R(x_0)$, the induction hypothesis ensures that $x_0\widehat{R}y\rightarrow_\Gmsf\hat{v}(\psi,y)=1$ for all $y\in\widehat{R}^+[x_0]$ whence, as $1\in\widehat{T}(x_0)$, $\hat{v}(\Box\psi,x_0)=1$. In case $v(\Box\psi,x_0)=\alpha_i<1$, by construction, $\alpha_i\in\widehat{T}(x)$. Then for all $y\in\widehat{W}$, $\alpha_i\leq \inf\lbrace x_0\widehat{R}y\rightarrow_\Gmsf\hat{v}(\psi,y)\mid y\in \widehat{W}\rbrace$. But given the choice of $y_{\Box\psi}\in\widehat{W}$, $x_0\widehat{R}y_{\Box\psi}\rightarrow_\Gmsf \hat{v}(\psi,y_{\Box\psi})<\alpha_{i+1}$. Then by construction of $\widehat{T}$, $\hat{v}(\Box\psi,x_0)=\alpha_i$. The case of $\lozenge\psi$ is similar.

Note, now, that $\lvert\widehat{W}\rvert\leq\lvert Y\rvert\cdot\lvert\sigma\rvert^n\leq\lvert\Sigma\rvert^{n+1}=\lvert\Sigma^{\hmc(\Mfrak)}\rvert$ and that as $\lvert \widehat{T}(x_0)\rvert\leq\lvert\Sigma_\Box\cup\Sigma_\Box^\invol\cup\Sigma_\lozenge\cup\Sigma_\lozenge^\invol\rvert+3$, $\lvert\widehat{T}(x_0)\rvert \leq\lvert\Sigma\rvert+1$. Finally, as $\langle\widehat{W},\widehat{R}\rangle\subseteq\langle W,R\rangle$, the new model inherits crispness.
\end{proof}

Theorem~\ref{theorem:semanticsequivalence} is now immediate from Lemmas~\ref{prop:Lemma1}, \ref{lemma:Lemma2}, and~\ref{lemma:Lemma3} as for any given $\phi$ there exists a~finite ICF containing~it.

We finish the section with a~short remark. Note from the proof Lemma~\ref{lemma:Lemma3} that we could not have assumed a~‘global’~$T$ (i.e., that $T(w)=T(w')$ for every $w,w'\in\Mfrak$). Indeed, we need to form $T(w)$'s according to the values of modal formulas in the corresponding states of the standard model. These values, however, are independent and thus, cannot be simulated by one~$T$.




\section{Tableaux\label{sec:tableaux}}
In Section~\ref{sec:Fmodels}, we used the result of~\cite{CaicedoMetcalfeRodriguezRogger2013} to obtain that $\KinvG$ has the finite model property with respect to $\Fmsf$-models. This, however, does not give a~decision algorithm. Thus, we will construct a tableaux calculus that allows us to extract countermodels from open branches and use it to define a~decision procedure that takes polynomial space. In~\cite{Rogger2016phd}, tableaux for $\KG$ are presented but they use the fact that all its connectives are order-based which is not the case for $\bimodalLinv$. To incorporate $\invol$, we follow~\cite{Haehnle1994} and provide a~\emph{constraint tableaux calculus}. We begin by giving formal definitions of the needed notions. We will then explain them in further detail.
\begin{definition}[Structures and constraints]\label{def:constraints}
We fix countable sets $\WorldLabels=\{w,w',w_0,\ldots\}$ of \emph{state-labels} and $\Var=\{c,d,e,c',\ldots\}$ of variables, define the set of $\Tmsf$-symbols $\Tmsf=\{t_i(w)\!\mid\!w\!\in\!\WorldLabels,i\!\in\!\Nmbb\}\cup\{t^s_i(w)\mid w\in\WorldLabels,i\in\Nmbb\}\cup\{\zero,\one\}$, and let $\triangledown\in\{\leqslant,<,\geqslant,>,=\}$.

We define sets of \emph{labelled formulas} ($\LF$) and \emph{relational terms} ($\relterm$) as follows:
\begin{align*}
\LF&=\{w:\phi\mid w\in\WorldLabels,\phi\in\bimodalLinv\}&\relterm&=\{w\Rmsf w'\mid w,w'\in\WorldLabels\}
\end{align*}

The set $\valueterm$ of \emph{value terms} and $\Str$ of \emph{structures} are defined as follows:
\begin{align*}
\valueterm\ni\Tmbb&\Coloneqq c\in\Var\mid\tmbf\in\Tmsf\mid\varrho\in\relterm\mid0\mid1\mid1-\Tmbb&\Str&=\valueterm\cup\LF
\end{align*}

Finally, \emph{constraints} have the form $\Sigma\triangledown\Tmbb$ s.t.\ $\Sigma\in\Str$, $\Tmbb\in\valueterm$.
\end{definition}
\begin{definition}[$\TKinvG$ --- tableaux for $\KinvG$]\label{def:TKinvG}
A \emph{tableau} is a downward-branching tree whose nodes are constraints. Each branch can be extended by one of the rules listed below. We apply the following conventions: vertical bars denote branching; $\blacktriangledown,\triangledown\!\in\!\{\leqslant,<,\geqslant,>\}$, if $\triangledown$ is $\leqslant$, then $\blacktriangledown$ is $\geqslant$ and vice versa (likewise for $<$); $\triangleright\in\{\geqslant,>\}$, $\triangleleft\in\{\leqslant,<\}$; $c$, $w'$, $t(w)$, and $t^s(w)$ are fresh on the branch; $w\Rmsf u$ occurs on the branch.
\begin{align*}
\invol:\dfrac{w:\invol\phi\triangledown\Tmbb}{w:\phi\blacktriangledown1\!-\!\Tmbb}
&&
\wedge_\triangleright:\dfrac{w\!:\!\phi\wedge\chi\triangleright\Tmbb}{\begin{matrix}w\!:\!\phi\triangleright\Tmbb\\w\!:\!\chi\triangleright\Tmbb\end{matrix}}
&&
\wedge_\triangleleft:\dfrac{w:\phi\wedge\chi\triangleleft\Tmbb}{w\!:\!\phi\triangleleft\Tmbb\mid w\!:\!\chi\triangleleft\Tmbb}\\[.4em]
\rightarrow_\triangleright:\dfrac{w:\phi\rightarrow\chi\triangleright\Tmbb}{w:\chi\triangleright\Tmbb\left|\begin{matrix}\Tmbb\triangleleft1\\w:\chi\geqslant c\\w:\phi\leqslant c\end{matrix}\right.}
&&
\rightarrow_\leqslant:\dfrac{w:\phi\rightarrow\chi\leqslant\Tmbb}{1\leqslant\Tmbb\left|\begin{matrix}w:\chi\leqslant\Tmbb\\w:\phi\geqslant c\\w:\chi<c\end{matrix}\right.}
&&
\rightarrow_<:\dfrac{w:\phi\rightarrow\chi<\Tmbb}{\begin{matrix}w:\chi<\Tmbb\\w:\phi\geqslant c\\w:\chi<c\end{matrix}}
\end{align*}
\begin{align*}
\Box_\triangleright:\dfrac{w:\Box\phi\triangleright\Tmbb}{\begin{matrix}w\!:\!\Box\phi\!=\!\one\\1\triangleright\!\Tmbb\end{matrix}\left|\begin{matrix}w\!:\!\Box\phi\!=\!t(w)\\\Tmbb\triangleleft t(w)\\w'\!:\!\phi\!>\!w\Rmsf w'\\w'\!:\!\phi\!<\!t^s(w)\end{matrix}\right.}
&&
\Box_\leqslant:\dfrac{w:\Box\phi\leqslant\Tmbb}{1\!\leqslant\!\Tmbb\left|\begin{matrix}\Tmbb\!\geqslant\!t(w)\\w'\!:\!\phi\!>\!w\Rmsf w'\\w'\!:\!\phi\!<\!t^s(w)\end{matrix}\right.}
&&
\Box_<:\dfrac{w\!:\!\Box\phi\!<\!\Tmbb}{\begin{matrix}\Tmbb\!>\!t(w)\\w'\!:\!\phi\!>\!w\Rmsf w'\\w'\!:\!\phi\!<\!t^s(w)\end{matrix}}\\[.4em]
\lozenge_\triangleleft:\dfrac{w\!:\!\lozenge\phi\!\triangleleft\!\Tmbb}{\begin{matrix}w\!:\!\lozenge\phi\!=\!\zero\\\Tmbb\!\triangleright\!0\end{matrix}\left|\begin{matrix}w\!:\!\lozenge\phi\!=\!t^s(w)\\t^s(w)\!\triangleleft\!\Tmbb\\w\Rmsf w'\!>\!t(w)\\w'\!:\!\phi\!>\!t(w)\end{matrix}\right.}
&&
\lozenge_\geqslant:\dfrac{w\!:\!\lozenge\phi\!\geqslant\!\Tmbb}{\Tmbb\!\leqslant\!0\left|\begin{matrix}\Tmbb\!\leqslant\!t^s(w)\\w\Rmsf w'\!>\!t(w)\\w'\!:\!\phi\!>\!t(w)\end{matrix}\right.}
&&
\lozenge_>:\dfrac{w\!:\!\lozenge\phi\!>\!\Tmbb}{\begin{matrix}\Tmbb\!<\!t^s(w)\\w\Rmsf w'\!>\!t(w)\\w'\!:\!\phi\!>\!t(w)\end{matrix}}\\[.4em]
\Box_=:\dfrac{w\!:\!\Box\phi\!=\!\Tmbb}{u\!:\!\phi\!\geqslant\!\Tmbb\left|\begin{matrix}u\!:\!\phi\!<\!\Tmbb\\u\!:\!\phi\!\leqslant\!w\Rmsf u\end{matrix}\right.}
&&
\lozenge_=:\dfrac{w\!:\!\lozenge\phi\!=\!\Tmbb}{w\Rmsf u\!\leqslant\!\Tmbb\left|\begin{matrix}u\!:\!\phi\!\leqslant\!\Tmbb\\w\Rmsf u\!>\!\Tmbb\end{matrix}\right.}
\end{align*}

Let $\Bmc=\{\cmc_1,\ldots,\cmc_n\}$ be a branch with constraints $\cmc_1$, \ldots, $\cmc_n$ and let further, $\cmc^\tmc$ be the result of replacing every $\Tmbb\in\valueterm$ and $\lambda\in\LF$ with variables $x_\Tmbb$ and $x_\lambda$, respectively, and every $\nmbb\in\{\zero,\one\}$ with $n\in\{0,1\}$. Furthermore, for every $w\in\WorldLabels$ occurring on $\Bmc$, we set
\begin{align*}
\Tmsf(w)&=\{t(w)\mid t(w)\text{ is on }\Bmc\}\cup\{t^s(w)\mid t^s(w)\text{ is on }\Bmc\}\cup\{\nmbb\mid\exists\phi\ w\!:\!\phi\!=\!\nmbb\!\in\!\Bmc\}
\end{align*}
and define $\Bmc$ to be \emph{closed} iff the following system of inequalities
\begin{align*}
\{\cmc^\tmc_1,\ldots,\cmc^\tmc_n\}\cup\{x_{t(w)}<x_{t^s(w)}\mid t(w)\text{ and }t^s(w)\text{ occur on }\Bmc\}
\end{align*}
\emph{does not have} a~solution over $[0,1]$ s.t.\ for every $w\in\WorldLabels$, the following properties hold:
\begin{align}\label{equ:closure}
\Tmsf(w)\!=\!\{t(w),t^s(w)\}\Rightarrow&\left[\begin{matrix}x_{t(w)}=0~\&~x_{t^s(w)}=\frac{1}{2}&\text{or }\\x_{t(w)}=\frac{1}{2}~\&~x_{t^s(w)}=1\end{matrix}\right]\nonumber\\[.4em]
|\Tmsf(w)|\geq3\Rightarrow&
\left[\begin{matrix}
\forall\tmbf\!\in\!\Tmsf(w)\;\exists\tmbf'\!\in\!\Tmsf(w):x_\tmbf=1-x_{\tmbf'}\text{ and}\\[.4em]
\exists\tmbf_1,\tmbf_2,\tmbf_3:x_{\tmbf_1}=0~\&~x_{\tmbf_2}=\frac{1}{2}~\&~x_{\tmbf_3}=1
\end{matrix}\right]\nonumber\\[.4em]
&\neg\exists t(w),t^s(w),t'(w):x_{t(w)}<x_{t'(w)}<x_{t^s(w)}
\end{align}

A~branch is \emph{open} if it is not closed. A~branch $\Bmc$ is \emph{complete} when for every premise of any rule occurring on~$\Bmc$, its conclusion also occurs in $\Bmc$. The only exceptions are branches containing constraints $w:\phi<\Tmbb$ and $w:\phi<1$ or $w:\phi>\Tmbb$ and $w:\phi>0$. In this case, the rules are applied to the constraints containing $\Tmbb$, not $0$ and~$1$.

Finally, $\phi\in\bimodalLinv$ \emph{has a~$\TKinvG$ proof} if there is a~tableau beginning with $w:\phi<1$ s.t.\ all its branches are closed.
\end{definition}
\begin{definition}[Model realising a~branch]\label{def:realisingmodel}
Let $\Mfrak=\langle W,R,T,v\rangle$ be an $\Fmsf$-model and $\Bmc$ a~tableau branch. An~\emph{$\Mfrak$-realisation of~$\Bmc$} is a~map $\real:\WorldLabels\cup\Str\rightarrow W\cup[0,1]$ s.t.\ $\real(w)\in W$ and $\real(\Sigma)\in[0,1]$ for every $\Sigma\in\Str$ and $w\in\WorldLabels$ occurring on~$\Bmc$, and the following properties hold:
\begin{enumerate}[noitemsep,topsep=2pt]
\item $\real(\zero)=0$ and $\real(\one)=1$;
\item $\real(w\Rmsf w')=\real(w)R\real(w')$;
\item if $\real(\Tmbb)=x$, then $\real(1-\Tmbb)=1-x$ for every $\Tmbb\in\valueterm$;
\item $(\{\real(\tmbf)\mid\tmbf\in\Tmsf(w)\}\cup\{0,\frac{1}{2},1\})\subseteq T(\real(w))$ and $\real(t(w))<\real(t^s(w))$ for all $w\in\WorldLabels$;
\item there are no $t'(w)$, $t(w)$, and $t^s(w)$ on~$\Bmc$ s.t.\ $\real(t(w))<\real(t'(w))<\real(t^s(w))$;
\item if $\Tmsf(w)=\{t(w),t^s(w)\}$, then $\frac{1}{2}\in\{\real(t(w)),\real(t^s(w))\}$;
\item if $|\Tmsf(w)|\geq3$, then for each $\tmbf\in\Tmsf(w)$ on $\Bmc$, there is $\tmbf'\in\Tmsf(w)$ on $\Bmc$ s.t.\ $\real(\tmbf)=1-\real(\tmbf')$, and there are $t_1,t_2,t_3\in\Tmsf(w)$ s.t.\ $\real(t_1)=0$, $\real(t_2)=\frac{1}{2}$, and $\real(t_3)=1$.
\end{enumerate}
A~constraint $w:\phi\triangledown\Tmbb$ is \emph{realised by~$\Mfrak$ under $\real$} if $v(\phi,\real(w))\triangledown\real(\Tmbb)$. A~constraint $\Tmbb\triangledown\Tmbb'$ is realised by~$\Mfrak$ under $\real$ if $\real(\Tmbb)\triangledown\real(\Tmbb')$. $\Bmc$ is realised by~$\Mfrak$ under $\real$ if $\real$ realises all constraints occurring on~$\Bmc$.
\end{definition}

Note that each application of a~modal rule to a~structure $w:\Box\phi$ or $w:\lozenge\chi$ that introduces a~witnessing state $w'$ to the branch also produces two value terms: $t(w)$ and $t^s(w)$. Each pair of such terms denotes the two consecutive elements of $T(w)$ between which $v(\phi,w')$ lies. 
Note also that to guarantee the soundness of such rules, we demand (cf.~items~4 and~5 in Definition~\ref{def:realisingmodel}) that the intervals do not overlap in a~non-trivial way. I.e., we prohibit the situations such as $\real(t_0(w))<\real(t_1(w))<\real(t^s_1(w))<\real(t^s_0(w))$. On the other hand, the intervals can coincide if this does not contradict other constraints on the branch.

For $\Box$ rules, this condition guarantees that the introduced $t(w)$ is indeed the maximal member of $T(w)$ that is still \emph{less or equal to the ‘real’ value of $\Box\phi$} according to the standard semantics. Dually, for $\lozenge$ rules, the condition ensures that the introduced $t^s(w)$ is the minimal member of $T(w)$ that is \emph{greater or equal to the value of $\lozenge\phi$} according to the standard semantics.

Let us now look at how the tableau proofs work. In the next example, we provide a~failed tableau proof of $\Box p\rightarrow\invol\lozenge\invol p$ and show how to read off countermodels from complete open branches.
\begin{example}\label{example:tableauproof}
Cf.~Fig.~\ref{fig:tableauproof}. The tableau has three closed branches marked with~$\times$ and a~complete open branch marked with $\frownie$. Let us check the closure of branches. In the leftmost branch, we use the fact that $\real(t(w))\in[0,1]$. Steps \textbf{17} and~\textbf{20} directly contradict each other. In the rightmost branch, we use that the values of $t(w)$'s should be closed under $1-x$. Namely, we have that $t_0(w)>0$ (otherwise, it would contradict $c\leq t_0(w)$ and $w:\invol\lozenge\invol p<c$). Hence, the values of $t(w)$'s are ordered as follows: $t_1(w)$, $t^s_1(w)$, $t_0(w)$, $t^s_0(w)$. As $\zero$ and~$\one$ do not occur on the branch, we have either (I)~$t_1(w)=0$, $t^s_1(w)=t_0(w)=\tfrac{1}{2}$, $t^s_0(w)=1$ or (II)~$t_0(w)=0$, $t^s_0(w)=t_1(w)=\tfrac{1}{2}$, $t^s_1(w)=1$. Observe, however, that $t^s_1(w)=t_0(w)=\tfrac{1}{2}$ in (I)~would contradict \textbf{16} and \textbf{22}, and $t_0(w)=0$ in (II)~would contradict \textbf{4} and \textbf{16}.

Thus, $w':p\geqslant t_0(w)$ and $w':p<1-t_1(w)$ cannot be realised at the same time.

Let us now discuss how to construct a~realising model of the open branch~$\frownie$. We note quickly that in the general case, each complete open branch can produce uncountably many models because constraints give order relations, not precise values.\footnote{Note that this is not a~problem from the practical standpoint as systems of linear inequalities over reals often have uncountably many solutions. Yet, they are algorithmically solvable.} Thus, we will provide just one example of a~countermodel.

Observe that $t(w)<wRw'\leqslant v(p,w')<1$. As there are no other constraints on $t(w)$ and $t^s(w)$, we put $\real(t(w))=0$, $\real(t^s(w))=\frac{1}{2}$, and $\real(\one)=1$. We now choose any value between $0$ and $1$ for $v(p,w')$ and set $wRw'=v(p,w')$ as the inequality is not strict. A~countermodel can be seen in Fig.~\ref{fig:tableauproof}. Using Definition~\ref{def:F-KinvG}, we see that $v(\Box p,w)\!=\!1$ and $v(\invol\lozenge\invol p,w)\!=\!\frac{1}{2}$ which falsifies $\Box p\!\rightarrow\!\invol\lozenge\invol p$ at~$w$.
\end{example}
\begin{figure}
\centering
\resizebox{1\linewidth}{!}{\begin{forest}
smullyan tableaux
[\textbf{1.}~w:\Box p\rightarrow\invol\lozenge\invol p<1
[\textbf{2.}~w:\invol\lozenge\invol p<1~(\rightarrow_<:\textbf{1})
[\textbf{3.}~w:\Box p\geqslant c~(\rightarrow_<:\textbf{1})
[\textbf{4.}~w:\invol\lozenge\invol p<c~(\rightarrow_<:\textbf{1})
[\textbf{5.}~w:\lozenge\invol p>1-c~(\invol:\textbf{4})
[\textbf{6.}~{w:\Box p=\one}~(\Box_\triangleright:\textbf{3})[\textbf{7.}~{1\geqslant c}~(\Box_\triangleright:\textbf{3})[\textbf{8.}~1-c<t^s(w)~(\lozenge_>:\textbf{5})[\textbf{9.}~t(w)<w\Rmsf w'~(\lozenge_>:\textbf{5})[\textbf{10.}~w':\invol p>t(w)~(\lozenge_>:\textbf{5})[\textbf{11.}~{w':p<1-t(w)}~(\invol:\textbf{10})
[\textbf{12.}~{w':p\geqslant\one}~({\Box_=:\textbf{6}})[\times:{\textbf{11},\textbf{12}}]][\textbf{13.}~w':p<\one~({\Box_=:\textbf{6}})[\textbf{14.}~w\Rmsf w'\leqslant w':p~({\Box_=:\textbf{6}})[\frownie]]]]]]]]]
[\textbf{15.}~{w:\Box p=t_0(w)}~(\Box_\triangleright:\textbf{3})[\textbf{16.}~{c\leqslant t_0(w)}~(\Box_\triangleright:\textbf{3})[\textbf{17.}~{w':p<w\Rmsf w'}~(\Box_\triangleright:\textbf{3})[\textbf{18.}~w':p<t^s_0(w)~(\Box_\triangleright:\textbf{3})[\textbf{19.}~w':p<t_0(w)~({\Box_=:\textbf{15}})[\textbf{20.}~{w':p\geqslant w\Rmsf w'}~({\Box_=:\textbf{15}})[\times:{\textbf{17},\textbf{20}}]]][\textbf{21.}~{w':p\geqslant t_0(w)}~({\Box_=:\textbf{15}})[\textbf{22.}~1-c<t^s_1(w)~({\lozenge_>:\textbf{5}})[\textbf{23.}~w'':\invol p>t_1(w)~({\lozenge_>:\textbf{5}})[\textbf{24.}~w\Rmsf w''>t_1(w)~({\lozenge_>:\textbf{5}})[\times:{\textbf{4},\textbf{16},\textbf{22}}]]]]]
]]]]]]]]]
\end{forest}}
\[\frownie:\xymatrix{w\ar^{R=\sfrac{1}{2}}[rr]&&w':p=\frac{1}{2}}\quad T(w)=\{0,\sfrac{1}{2},1\}\]
\caption{A~failed $\TKinvG$ proof of $\Box p\rightarrow\invol\lozenge\invol p$ and a~model of the open branch (marked with $\frownie$).}
\label{fig:tableauproof}
\end{figure}

We are now ready to show the soundness and completeness of our tableaux calculus.
\begin{restatable}{theorem}{TKinvGcompleteness}\label{theorem:TKinvGcompleteness}
$\phi\in\bimodalLinv$ is $\KinvG$-valid iff it has a~$\TKinvG$ proof.
\end{restatable}
\begin{proof}
The proof is standard and follows~\cite[Theorems~6.19 and~6.21]{Rogger2016phd}, so we only give a~sketch thereof. For the soundness part, we show that if a~tableau beginning with $w:\phi<1$ is closed, then $\phi$ is $\KinvG$-valid. We show that if $\Mfrak$ realises the premises of a~rule, then it should also realise at least one of its conclusions (note that branching rules have two conclusions). The soundness will follow since there is no $\Fmsf$-model~$\Mfrak$ and realisation $\real$ s.t.\ $\Mfrak$ realises a~closed branch~$\Bmc$ under~$\real$.

We consider the case of the $\Box_\leqslant$ rule as an example. Let $\Mfrak=\langle W,R,T,v\rangle$ realise $w:\Box\phi\leqslant\Tmbb$. If $\real(\Tmbb)=1$, then $1\leqslant\Tmbb$ (the left conclusion) is realised. Consider the case where $\real(\Tmbb)<1$. Then (cf.~Definition~\ref{def:F-KinvG}) $\max\{x\in T(w)\mid x\leq\inf\limits_{w'\in W}\{\real(w)Rw'\rightarrow_\Gmsf v(\phi,w')\}\}\leq\real(\Tmbb)$. This means that there is a~state $w'$ s.t.\ $v(\phi,w')<wRw'$ and there are $x,x'\in T(w)$ s.t.\ $x'$ is the immediate successor of~$x$, $v(\Box\phi,w)=x$, and $v(\phi,w')<x'$. Now setting $\real(t(w))=x$, $\real(t^s(w))=x'$, and $\real(w')=w'$, we have that $t(w)\leqslant\Tmbb$ and $w':\phi<t^s(w)$ are realised. Hence, the right conclusion is realised.

For the completeness part, we prove that for every complete open branch, there is an $\Fmsf$-model $\Mfrak$ and a~realisation $\real$ s.t.\ $\Mfrak$ realises $\Bmc$ under $\real$. Let~$\Bmc$ be a~complete open branch and consider the system of inequalities $\Bmc^\tmc$ as shown in~\eqref{equ:closure} and its solution as specified in Definition~\ref{def:TKinvG}. Now define $\Mfrak$ as follows: $W=\{w\mid w\text{ occurs in }\Bmc\}$, $wRw'=x_{w\Rmsf w'}$, $T(w)=\{x_{t(w)}\mid t(w)\text{ occurs in }\Bmc\}\cup\{x_{t^s(w)}\mid t^s(w)\text{ occurs in }\Bmc\}\cup\{0,\frac{1}{2},1\}$ and $v(p,w)=x_{w:p}$ for every $w$, $w'$, and $p\in\Prop$ occurring in~$\Bmc$. We also define $\real(w)=w$, $\real(\Tmbb)=x_\Tmbb$, and $\real(w:\chi)=x_{w:\chi}$ for every $w\in\WorldLabels$, $\Tmbb\in\valueterm$, and $w:\chi\in\LF$ occurring on~$\Bmc$. It remains to show that all constraints are realised by $\Mfrak$ under $\real$.

For all constraints of the form $\Tmbb\triangledown\Tmbb'$ that do not contain labelled formulas, we have $\real(\Tmbb)\triangledown\real(\Tmbb')$ by the construction of~$\Mfrak$ since $\Bmc$ is a~complete open branch. Let us now show that all constraints $u:\psi\triangledown\Tmbb$ on $\Bmc$ are realised. We proceed by induction on formulas. Constraints of the form $w:p\triangledown\Tmbb$ are realised by construction of~$\Mfrak$. The cases of constraints whose principal connective is $\invol$, $\wedge$, or $\rightarrow$ can be shown by simple applications of the induction hypothesis.

Let us consider the case of $w:\Box\phi\leq\Tmbb$ as an example and show that $v(\Box\phi,w)\leq\real(\Tmbb)$. Since $w:\Box\phi\leqslant\Tmbb$ is in~$\Bmc$ and $\Bmc$ is complete and open, then it also contains (a)~$1\leqslant\Tmbb$ or (b) $t(w)\leqslant\Tmbb$, $w':\phi<w\Rmsf w'$, and $w':\phi<t^s(w)$. If (a) is the case, $\Tmbb$ is a~value term (observe from Definition~\ref{def:constraints} that constraints can contain at most one labelled formula), whence $1\leqslant\Tmbb$ is realised by the construction of $\Mfrak$. Thus, $\real(\Tmbb)=1$ and $v(\Box\phi,w)\leq\real(\Tmbb)$, as required. If (b) is the case, then by the induction hypothesis $t(w)\leqslant\Tmbb$, $w':\phi<w\Rmsf w'$, and $w':\phi<t^s(w)$ are realised, whence, $v(\phi,w')<wRw'$, $x_{t(w)}\leq\real(\Tmbb)$, and $v(\phi,w')<x_{t^s(w)}$ with $\{x_{t(w)},x_{t^s(w)}\}\subseteq T(w)$ and $x_{t^s(w)}$ being the immediate successor of $x_{t(w)}$. But in this case, it is clear that $v(\Box\phi,w)=\max\{x\in T(w)\mid x\leq\inf\limits_{w'\in W}\{wRw'\rightarrow_\Gmsf v(\phi,w')\}\}\leq x_{t(w)}$, whence, $v(\Box\phi,w)\leq\real(\Tmbb)$, as required.
\end{proof}
\section{Complexity\label{sec:complexity}}
Let us now use $\TKinvG$ to obtain the $\pspace$-completeness of $\KinvG$. Our proof adapts the standard algorithm for $\mathbf{K}$ from~\cite{BlackburndeRijkeVenema2010}.
\begin{restatable}{theorem}{pspacecompleteness}\label{theorem:pspacecompleteness}
Validity of $\KinvG$ is $\pspace$-complete.
\end{restatable}
\begin{proof}
$\pspace$-hardness is immediate since $\KG$ is $\pspace$-hard~\cite[Theorem~21]{CaicedoMetcalfeRodriguezRogger2017} and $\KinvG$ is a~conservative extension of $\KG$. Let us now tackle the upper bound.

First of all, we observe that all tableaux terminate. Indeed, the branching factor of rules is at most~$2$, and every rule except for $\Box_\triangleright$ and $\lozenge_\triangleleft$ reduces the number of connectives in the labelled formula in the premise. If $\Box_\triangleright$ or $\lozenge_\triangleleft$ is applied, then $w:\Box\phi=t(w)$ or $w:\lozenge\phi=t^s(w)$, respectively, is introduced to which one can apply $\Box_=$ or $\lozenge_=$ which does decompose the formula. Thus, at some point, all formulas will be decomposed into atoms. We note, moreover, that given a~branch~$\Bmc$, it takes us nondeterministic polynomial time to find a~solution to $\Bmc^\tmc$ over $[0,1]$ s.t.~\eqref{equ:closure} holds with respect to~it.

Let us now provide a~decision procedure that uses only polynomial space with respect to~$\lmc(\phi)$. We aim to build the model realising $w^0_1:\phi<1$ ‘on the fly’ (in which case, $\phi$ is not valid) or to show that it is impossible. Fig.~\ref{fig:algorithm} illustrates the process of building the model. As one sees, we \emph{do not} build model branch-by-branch but rather generate all children of a given node, then pick one child, and again, generate all its children. This is because we need to apply rules $\Box_=$ and $\lozenge_=$ using generated states \emph{all at once}. But these rules can be applied \emph{only after} $\Box_\triangleright$ and $\lozenge_\triangleleft$.

We begin with $w^0_1:\phi<1$. In what follows, if a~rule introduces branching to the tableau, we pick one branch, say, $\Bmc$, and work with it depth-first. If the branch with which we are working is closed, we delete it and choose the next one. First, we apply propositional rules. Then we apply modal rules for constraints containing $\leqslant$ or $<$. This adds new states $w^1_{1,1}$, \ldots, $w^1_{n^0_1,1}$ (we will call them \emph{children} of~$w^0_1$), new members of $\Tmsf(w^0_1)$ of the form $t_i(w^0_1)$, $t^s_i(w^0_1)$, $\zero$, and $\one$, relational terms $w^0_1\Rmsf w^1_{i,1}$, and new constraints of the form $\Box\chi=t_i(w^0_1)$ or $\lozenge\chi=t^s_i(w^0_1)$ for $i\in\{1,\ldots,n^0_1\}$. Note that $n^0_1=\Omc(\lmc(\phi))$, whence, we need only $\Omc(\lmc(\phi))$ space to store the ‘content’ of $w^0_1$. 

If $\Bmc$ is still open, mark $w^1_{1,1}$ as ‘active’ and apply all rules for constraints of the form $\Box\chi=t_i(w^0_1)$ and $\lozenge\chi=t^s_i(w^0_1)$ using $w^1_{1,1}$ in the conclusion. Note that $\Box_=$ and $\lozenge_=$ are the only modal rules that can become applicable after the application of $\Box_\triangleright$ and $\lozenge_\triangleleft$. Once all $\Box_=$ and $\lozenge_=$ rules are applied, we apply the propositional rules to formulas labelled with $w^1_{1,1}$. We then apply modal rules for constraints with $\leqslant$ and~$<$. As on the previous stage, we generate new states $w^2_{1,1}$, \ldots, $w^2_{n^1_2,1}$, constraints of the form $\Box\chi=t_i(w^1_{1,1})$ or $\lozenge\chi=t^s_i(w^1_{1,1})$, and members of $\Tmsf(w^1_{1,1})$ and relational terms corresponding to them. We repeat the procedure until we produce a~state $w^m_{1,1}$ s.t.\ all constraints $w^m_{1,1}:\chi\triangledown\Tmbb$ are atomic. If the branch is not closed, we mark $w^m_{1,1}$ as ‘safe’, delete all $w^m_{1,1}:\chi\triangledown\Tmbb$'s and go to $w^m_{2,1}$ and repeat the procedure. Once all states $w^m_{1,1}$, \ldots, $w^m_{n^1_m,1}$ are marked as ‘safe’, we delete constraints containing them, mark $w^{m-1}_{1,1}$ as ‘safe’, and delete all its children. Then we proceed to $w^{m-1}_{2,1}$. The procedure is repeated until either all tableau branches are closed (which means that $\phi$ is valid) or $w^0_1$ is marked as ‘safe’ (and thus, we have a~complete open branch, whence, $\phi$ is not valid).

Note that the length of the branch of the constructed model is bounded from above by the modal depth of $\phi$ and is thus $\Omc(\lmc(\phi))$. In addition, each state of the branch of the model can have at most $\Omc(\lmc(\phi))$ children because their number is bounded from above by the number of modalities on the same level of nesting. Thus, we have to store $\Omc((\lmc(\phi))^2)$ states each of which contains constraints taking up to $\Omc(\lmc(\phi))$ space. Thus, our decision procedure utilises $\Omc((\lmc(\phi))^3)$ space.
\end{proof}
\begin{figure}
\centering
\resizebox{.9\linewidth}{!}{
\begin{tikzpicture}[>=stealth,relative]
\node (w01initial) at (-1,-.5) {\textcolor{red}{$w^0_1$}};
\node (w01stage1) at (2,0) {$w^0_1$};
\node (w111stage1) at (1,-1) {\textcolor{red}{$w^1_{1,1}$}};
\node (w121stage1) at (2,-1) {$w^1_{2,1}$};
\node at (3,-1) {$\ldots$};
\node (w1n01stage1) at (4,-1) {$w^1_{n^0_1,1}$};
\path[->,draw] (w01stage1) to (w111stage1);
\path[->,draw] (w01stage1) to (w121stage1);
\path[->,draw] (w01stage1) to (w1n01stage1);
\node (initial) at (-.75,-.5) {};
\node (stage1to) at (.75,-.5) {};
\draw[double,->,draw=blue] (initial) to (stage1to);
\node (stage1from) at (4.25,-.5) {};
\node (w01stage2) at (7,0) {$w^0_1$};
\node (w111stage2) at (6,-1) {$w^1_{1,1}$};
\node (w121stage2) at (7,-1) {$w^1_{2,1}$};
\node at (8,-1) {$\ldots$};
\node (w1n01stage2) at (9,-1) {$w^1_{n^0_1,1}$};
\node (w211stage2) at (5,-2) {\textcolor{red}{$w^2_{1,1}$}};
\node (w221stage2) at (6,-2) {$w^2_{2,1}$};
\node at (7,-2) {$\ldots$};
\node (w2n111stage2) at (8,-2) {$w^2_{n^1_2,1}$};
\path[->,draw] (w01stage2) to (w111stage2);
\path[->,draw] (w01stage2) to (w121stage2);
\path[->,draw] (w01stage2) to (w1n01stage2);
\path[->,draw] (w111stage2) to (w211stage2);
\path[->,draw] (w111stage2) to (w221stage2);
\path[->,draw] (w111stage2) to (w2n111stage2);
\node (stage2to) at (6,-.5) {};
\draw[double,->,draw=blue] (stage1from) to (stage2to);
\node (w01stage3) at (12,0) {$w^0_1$};
\node (w111stage3) at (11,-1) {$w^1_{1,1}$};
\node (w121stage3) at (12,-1) {$w^1_{2,1}$};
\node at (13,-1) {$\ldots$};
\node (w1n01stage3) at (14,-1) {$w^1_{n^0_1,1}$};
\node (w211stage3) at (10,-2) {$w^2_{1,1}$};
\node (w221stage3) at (11,-2) {$w^2_{2,1}$};
\node at (12,-2) {$\ldots$};
\node (w2n111stage3) at (13,-2) {$w^2_{n^1_2,1}$};
\path[->,draw] (w01stage3) to (w111stage3);
\path[->,draw] (w01stage3) to (w121stage3);
\path[->,draw] (w01stage3) to (w1n01stage3);
\path[->,draw] (w111stage3) to (w211stage3);
\path[->,draw] (w111stage3) to (w221stage3);
\path[->,draw] (w111stage3) to (w2n111stage3);
\node at (10,-2.4) {$\vdots$};
\node (wpenultimatestage3) at (10,-3) {$w^{m-1}_{1,1}$};
\node (wm11stage3) at (10,-4) {\textcolor{green}{$w^m_{1,1}$}};
\node (wm21stage3) at (11,-4) {$w^m_{2,1}$};
\node at (12,-4) {\ldots};
\node (wmnm1stage3) at (13,-4) {$w^m_{n^1_m,1}$};
\path[->,draw] (wpenultimatestage3) to (wm11stage3);
\path[->,draw] (wpenultimatestage3) to (wm21stage3);
\path[->,draw] (wpenultimatestage3) to (wmnm1stage3);
\node (stage2from) at (8.5,-.5) {};
\node (stage3to) at (11,-.5) {};
\draw[double,->,dashed,draw=blue] (stage2from) to (stage3to);
\node (w01stage4) at (-1,-5) {$w^0_1$};
\node (w111stage4) at (-2,-6) {$w^1_{1,1}$};
\node (w121stage4) at (-1,-6) {$w^1_{2,1}$};
\node at (0,-6) {$\ldots$};
\node (w1n01stage4) at (1,-6) {$w^1_{n^0_1,1}$};
\node (w211stage4) at (-3,-7) {$w^2_{1,1}$};
\node (w221stage4) at (-2,-7) {$w^2_{2,1}$};
\node at (-1,-7) {$\ldots$};
\node (w2n111stage4) at (0,-7) {$w^2_{n^1_2,1}$};
\path[->,draw] (w01stage4) to (w111stage4);
\path[->,draw] (w01stage4) to (w121stage4);
\path[->,draw] (w01stage4) to (w1n01stage4);
\path[->,draw] (w111stage4) to (w211stage4);
\path[->,draw] (w111stage4) to (w221stage4);
\path[->,draw] (w111stage4) to (w2n111stage4);
\node at (-3,-7.4) {$\vdots$};
\node (wpenultimatestage4) at (-3,-8) {$w^{m-1}_{1,1}$};
\node (wm11stage4) at (-3,-9) {\textcolor{green}{$w^m_{2,1}$}};
\node (wm21stage4) at (-2,-9) {\textcolor{red}{$w^m_{2,1}$}};
\node at (-1,-9) {\ldots};
\node (wmnm1stage4) at (0,-9) {$w^m_{n^1_m,1}$};
\path[->,draw] (wpenultimatestage4) to (wm11stage4);
\path[->,draw] (wpenultimatestage4) to (wm21stage4);
\path[->,draw] (wpenultimatestage4) to (wmnm1stage4);
\node (stage3from) at (9.5,-3.5) {};
\node (stage4to) at (-0.5,-4.5) {};
\draw[double,->,draw=blue] (stage3from) to (stage4to);
\node(stage4from) at (0,-5) {};
\node(stage5to) at (3.5,-5) {};
\draw[double,->,dashed,draw=blue] (stage4from) to (stage5to);
\node (w01stage5) at (4,-5) {$w^0_1$};
\node (w111stage5) at (3,-6) {$w^1_{1,1}$};
\node (w121stage5) at (4,-6) {$w^1_{2,1}$};
\node at (5,-6) {$\ldots$};
\node (w1n01stage5) at (6,-6) {$w^1_{n^0_1,1}$};
\node (w211stage5) at (2,-7) {$w^2_{1,1}$};
\node (w221stage5) at (3,-7) {$w^2_{2,1}$};
\node at (4,-7) {$\ldots$};
\node (w2n111stage5) at (5,-7) {$w^2_{n^1_2,1}$};
\path[->,draw] (w01stage5) to (w111stage5);
\path[->,draw] (w01stage5) to (w121stage5);
\path[->,draw] (w01stage5) to (w1n01stage5);
\path[->,draw] (w111stage5) to (w211stage5);
\path[->,draw] (w111stage5) to (w221stage5);
\path[->,draw] (w111stage5) to (w2n111stage5);
\node at (2,-7.4) {$\vdots$};
\node (wpenultimatestage5) at (2,-8) {$w^{m-1}_{1,1}$};
\node (wm11stage5) at (2,-9) {\textcolor{green}{$w^m_{1,1}$}};
\node (wm21stage5) at (3,-9) {\textcolor{green}{$w^m_{2,1}$}};
\node at (4,-9) {\ldots};
\node (wmnm1stage5) at (5,-9) {\textcolor{green}{$w^m_{n^1_m,1}$}};
\path[->,draw] (wpenultimatestage5) to (wm11stage5);
\path[->,draw] (wpenultimatestage5) to (wm21stage5);
\path[->,draw] (wpenultimatestage5) to (wmnm1stage5);
\node (w01stage6) at (9,-5) {$w^0_1$};
\node (w111stage6) at (8,-6) {$w^1_{1,1}$};
\node (w121stage6) at (9,-6) {$w^1_{2,1}$};
\node at (10,-6) {$\ldots$};
\node (w1n01stage6) at (11,-6) {$w^1_{n^0_1,1}$};
\node (w211stage6) at (7,-7) {$w^2_{1,1}$};
\node (w221stage6) at (8,-7) {$w^2_{2,1}$};
\node at (9,-7) {$\ldots$};
\node (w2n111stage6) at (10,-7) {$w^2_{n^1_2,1}$};
\path[->,draw] (w01stage6) to (w111stage6);
\path[->,draw] (w01stage6) to (w121stage6);
\path[->,draw] (w01stage6) to (w1n01stage6);
\path[->,draw] (w111stage6) to (w211stage6);
\path[->,draw] (w111stage6) to (w221stage6);
\path[->,draw] (w111stage6) to (w2n111stage6);
\node at (7,-7.4) {$\vdots$};
\node (wpenultimatestage4) at (7,-8) {\textcolor{green}{$w^{m-1}_{1,1}$}};
\node (stage5from) at (4.5,-5) {};
\node (stage6to) at (8.5,-5) {};
\draw[double,->,draw=blue] (stage5from) to (stage6to);
\end{tikzpicture}}
\caption{Building the model on the fly: \textcolor{red}{red} states are active; \textcolor{green}{green} states are safe.}
\label{fig:algorithm}
\end{figure}
\section{Conclusion\label{sec:conclusion}}
We presented an expansion of the G\"{o}del modal logic with the involutive negation and showed how one can use it to formalise reasoning about uncertainty. We also constructed an alternative semantics for $\KinvG$ following~\cite{CaicedoMetcalfeRodriguezRogger2013,CaicedoMetcalfeRodriguezRogger2017} with respect to which $\KinvG$ has the finite model property. We then used these semantics to provide the tableaux calculus $\TKinvG$ that allows us to read off finite countermodels from complete open branches. Finally, we proved $\pspace$-completeness of $\KinvG$-validity using $\TKinvG$.

Our next steps are as follows. First, we want to continue the systematic study of $\KinvG$. In particular, we plan to provide a~complete Hilbert-style axiomatisation of $\KinvG$ over crisp and fuzzy frames. This will help investigate the correspondence between the classes of crisp and fuzzy frames on the one side and the $\bimodalLinv$-formulas that define these classes of frames, on the other. Furthermore, in~\cite{MetcalfeOlivetti2009,MetcalfeOlivetti2011}, hypersequent calculi for $\Box$- and $\lozenge$-fragments of $\KG$ are presented. Since these calculi are cut-free, they can serve as a~syntactical decision procedure (i.e., the one that does not directly produce a~countermodel) for the fragments of~$\KG$. It thus makes sense to construct hypersequent cut-free calculi for the bi-modal $\KG$ and its expansions and obtain an alternative, syntactical, decision procedure for~$\KG$.

Second, in~\cite{BilkovaFrittellaKozhemiachenko2022IJCAR,BilkovaFrittellaKozhemiachenko2023IGPL,BilkovaFrittellaKozhemiachenko2024JLC}, we considered $\KGsquare$ --- an expansion of $\KG$ with a~so-called ‘strong negation’ from~\cite{Wansing2008} with semantics defined on $[0,1]^{\Join}$ --- a~bi-lattice product of $[0,1]$ with itself. In particular, it was shown that a~class of frames is $\KGsquare$-definable iff it is $\KG$-definable by a~self-dual formula (i.e., a~formula equivalent to the negation of the result of replacing every connective with its dual: $\wedge$ with $\vee$, $\rightarrow$ with~$\coimplies$, and $\Box$ with $\lozenge$). Note, however, that $\KGsquare$ does not introduce connectives defined w.r.t.\ the second order on~$[0,1]^{\Join}$. On the other hand, in~\cite{Speranski2022}, it is shown that the modal logic on $\{0,1\}^{\Join}$ in the language with all bi-lattice connectives and modalities and the bi-modal classical modal logic have the same expressivity. It would be thus instructive to consider a~full bi-lattice expansion of $\KGsquare$ with semantics on $[0,1]^{\Join}$ and compare its expressivity to $\KG$ and its expansions. 
\bibliographystyle{eptcs}
\bibliography{generic}
\end{document}